\documentclass[10pt,oneside,reqno]{amsart}
\usepackage[T1]{fontenc}
\usepackage{amsmath,amsthm, mathtools} 
\usepackage{enumerate}
\usepackage{xfrac}          

\usepackage{comment}
\usepackage[backend = biber, style=alphabetic, citestyle=alphabetic, giveninits=true
]{biblatex}
	\renewbibmacro{in:}{}
	\AtEveryBibitem{\clearlist{language}}
	
\usepackage{halloweenmath}

\numberwithin{equation}{section}
\theoremstyle{definition}
\newtheorem{Definition}{Definition}[section]
\newtheorem{Example}[Definition]{Example}
\newtheorem{Remark}[Definition]{Remark}

\theoremstyle{plain}
\newtheorem{Theorem}[Definition]{Theorem}

\newtheorem{Proposition}[Definition]{Proposition}

\newtheorem{Lemma}[Definition]{Lemma}

\usepackage{apptools}
\AtAppendix{\renewtheorem{definition}{Definition}[section]
\renewtheorem{proposition}[definition]{Proposition}
\renewtheorem{lemma}[definition]{Lemma}}

\usepackage{amssymb}        
\usepackage{mathrsfs}       
\usepackage{eucal}          
\usepackage{stmaryrd}       

\usepackage{geometry}

\usepackage{hyperref}

\usepackage[usenames,dvipsnames]{xcolor}
\usepackage{tikz}
\usetikzlibrary{arrows, matrix}
\usepackage{tikz-cd}
\usepackage{subfig}         
\usepackage{arydshln}       

\newcommand{\bDiamond}{\mathbin{\Diamond}}

\makeatletter
\newcommand\bigDiamond{\mathop{\mathpalette\bigDi@mond\relax}}
\newcommand\bigDi@mond[2]{%
	\vcenter{\hbox{\m@th
			\scalebox{\ifx#1\displaystyle 2\else1.2\fi}{$#1\Diamond$}%
	}}%
}
\makeatother

\usepackage{romanbar}

\newcommand{\catname}[1]{\mathbf{#1}}

\newcommand{\al}{\alpha}
\newcommand{\be}{\beta}

\newcommand{\la}{\lambda}

\newcommand{\si}{\sigma}

\newcommand{\Z}{\mathbb{Z}}

\newcommand{\R}{\mathbb{R}}
\newcommand{\C}{\mathbb{C}}

\newcommand{\Fgl}{\mathfrak{gl}}
\newcommand{\Fsl}{\mathfrak{sl}}

\newcommand{\Fosp}{\mathfrak{osp}}

\newcommand{\Fg}{\mathfrak{g}}
\newcommand{\Fh}{\mathfrak{h}}
\newcommand{\Fk}{\mathfrak{k}}
\newcommand{\Fl}{\mathfrak{l}}

\newcommand{\Fn}{\mathfrak{n}}

\newcommand{\Fp}{\mathfrak{p}}

\newcommand{\FG}{\mathfrak{G}}
\newcommand{\FH}{\mathfrak{H}}
\newcommand{\FN}{\mathfrak{N}}

\newcommand{\CA}{\mathcal{A}}

\newcommand{\CC}{\mathcal{C}}

\newcommand{\CL}{\mathcal{L}}

\newcommand{\op}{\operatorname}

\DeclareMathOperator{\End}{End}

\DeclareMathOperator{\Id}{Id}

\DeclareMathOperator{\U}{U}

\renewcommand{\hat}{\widehat}
\renewcommand{\tilde}{\widetilde}
\renewcommand{\mod}[1]{\;\text{(mod $#1$)}}

\newcommand{\antiZ}{\reflectbox{$Z$}}
\newcommand{\ghostZ}{\mathghost} 

\makeatletter
\def\ifemptyarg#1{%
	\if\relax\detokenize{#1}\relax 
	\expandafter\@firstoftwo
	\else
	\expandafter\@secondoftwo
	\fi}
\makeatother

\makeatletter
\def\ifemptyarg#1{%
		\if\relax\detokenize{#1}\relax
	\expandafter\@firstoftwo
	\else
	\expandafter\@secondoftwo
	\fi}
\makeatother

\usepackage{stackrel}

\title{Ghost center and representations of the diagonal reduction algebra of $\Fosp(1|2)$}
\author{Jonas T. Hartwig \and Dwight Anderson Williams II}

\date{March 30, 2022}

\address{Department of Mathematics, Iowa State University, Ames, IA-50011, USA}
\email{jth@iastate.edu}
\urladdr{http://jthartwig.net}
\address{MathDwight, The Bronx, New York, NY-10462, USA}
\email{dwight@mathdwight.com}
\urladdr{https://mathdwight.com}

\begin{document}
\maketitle
\begin{abstract}
Reduction algebras are known by many names in the literature, including step algebras, Mickelsson algebras, Zhelobenko algebras, and transvector algebras, to name a few.
These algebras, realized by raising and lowering operators, allow for the calculation of Clebsch-Gordan coefficients, branching rules, and intertwining operators; and have connections to extremal equations and dynamical R-matrices in integrable face models.

In this paper we continue the study of the diagonal reduction superalgebra $A$ of the orthosymplectic Lie superalgebra $\Fosp(1|2)$. We construct a Harish-Chandra homomorphism, Verma modules, and study the Shapovalov form on each Verma module. Using these results, we prove that the ghost center (center plus anti-center) of $A$ is generated by two central elements and one anti-central element  (analogous to the Scasimir due to Le\'{s}niewski for $\Fosp(1|2)$). As another application, we classify all finite-dimensional irreducible representations of $A$. Lastly, we calculate an infinite-dimensional tensor product decomposition explicitly.         
\end{abstract}


\section{Introduction}

The structure and representations of reduction algebras \cite{zhelobenkoRepresentationsReductiveLie1994} are useful for solving problems in representation theory of Lie algebras.
Specifically, they are designed as a tool for the construction of explicit intertwining operators between representations related to branching rules. 
But reduction algebras, and the closely connected notion of extremal projectors \cite{asherovaProjectionOperatorsSimple1973}, have also been shown to have a natural interplay with solutions to the Arnaudon-Buffenoir-Ragoucy-Roche (ABRR) equation \cite{khoroshkinExtremalProjectorDynamical2004}. The ABRR equation in turn can be used to create solutions to the dynamical Yang-Baxter equation, which expresses the integrability of face models in statistical mechanics \cite{Felder:1994be,felder1995conformal}.

Another application of these algebras we wish to highlight is that they can be represented by linear operators on spaces of solutions to many systems of equations which admit the superposition principle, including (relativistic) wave equations like the Dirac equation and Maxwell's equations, or other equations governing physical phenomena. 
Zhelobenko \cite{zhelobenkoHypersymmetriesExtremalEquations1997} provides a guide  to the general construction of a representation of a reduction algebra where the representation space appears as the solution space to a system of equations:

Begin with a system of equations $\epsilon v = 0$ of which solutions $v$ live in the solution space $V^{\epsilon}$ with $\epsilon \subset \CA$ denoting a subset of operators in an associative algebra $\CA$. That is, $e v = 0$ for every $e \in \epsilon$. Take for example $\epsilon$ to be the singelton containing the Laplacian $\Delta$ and the $m$th Weyl algebra $A_{m}$ as $\CA$ so that $V^{\epsilon}$ is the space of harmonic polynomials within the irreducible $A_{m}$-module $V = \R[x_{1}, x_{2}, \ldots x_{m}]$. The left ideal $I_{\epsilon} \subset \CA$ generated by $\epsilon$ acts in $V^{\epsilon}$ by annihilation. Furthermore, if we let $N(I_{\epsilon})$ be the largest subalgebra of $\CA$ in which  $I_{\epsilon}$ is a two-sided ideal (equivalently, $N(I_{\epsilon}) = \{n \in \CA \mid \epsilon n \subset I_{\epsilon}  \}$ is the normalizer of $I_{\epsilon}$ in $\CA$), then the space of solutions is stable under the action of the quotient algebra $N(I_{\epsilon})/I_{\epsilon}$. Indeed, $\epsilon(n+I_{\epsilon})v \subset I_{\epsilon}v = 0$, for all $n$ in $N(I_{\epsilon})$. The steps above form what Zhelobenko calls the \emph{translator algebra} $N(I_{\epsilon})/I_{\epsilon}$ or the Mickelsson algebra \cite{mickelssonStepAlgebrasSemisimple1973,vandenhomberghNoteMickelssonStep1975,van1976harisch} when $\epsilon$ is the span of the positive root vectors in a semi-simple Lie algebra $\Fg \xhookrightarrow{} U(\Fg)  \xhookrightarrow{}\CA$, equating $V^{\epsilon}$ with the space $V^{+}$ of singular (also known as primitive) vectors. These are not entirely separate considerations. Continuing the example of $\epsilon = \{\Delta\}$, one can embed $U(\Fsl(2))$ into a Weyl algebra using an $\Fsl(2)$-triple of linearly transformed versions of the Laplacian, the Euler operator, and the squared norm in $\R^{m}$. 

One defines the reduction algebra to be a localized version of the Mickelsson algebra \cite{zhelobenkoRepresentationsReductiveLie1994}. This is done to facilitate the utilization of the extremal projector \cite{asherovaProjectionOperatorsSimple1973}. The process of passing from $V$ to $V^\epsilon$ is categorical in nature in that, under cases we consider, the space of solutions $V^{\epsilon}$ is an irreducible module of the reduction algebra whenever $V$ is an irreducible $\CA$-module. In the non-localized setting of enveloping algebras this fact was first conjectured by Mickelsson \cite{mickelssonStepAlgebrasSemisimple1973} and proved by van den Hombergh \cite{vandenhomberghNoteMickelssonStep1975}. Returning to $\epsilon = \{\Delta\}$, we conclude that the space of harmonic polynomials is an irreducible module (generated by $f \equiv 1$) over the reduction algebra described above by $\Fsl(2)  \xhookrightarrow{} U(\Fsl(2))  \xhookrightarrow{} A_{m}$.

Excitingly, the study of reduction \emph{super}algebras also achieves results in the representation theory of Lie \emph{super}algebras \cite{matsumotoRepresentationsCentrallyExtended2014}. For example, in this paper we illustrate how one can use reduction algebras to decompose certain (infinite-dimensional) tensor product representations of $\Fosp(1|2)$. One can also recover tables of Clebsch-Gordan coefficients of the Lie superalgebra $\Fosp(1|2)$ \cite{scheunertIrreducibleRepresentationsOsp1977,bergeronGeneratingFunctionsMathfrakosp2016} in this way. As an application of extremal projectors, \cite{Tolstoy:2004tp} gives an overview of reduction superalgebras that have been described for some low corank reductive pairs of Lie superalgebras; however, keeping in mind the opening paragraphs and the connections that go beyond the intrigue of branching rules, representations of a given reduction (super)algebra intersect a broad interest and should be considered independently of their origin.
Yet it seems that only recently the structure and representation theory of reduction superalgebras have been investigated in the literature in a manner comparable to the general linear algebra case found in \cite{khoroshkinMickelssonAlgebrasZhelobenko2008, khoroshkinStructureConstantsDiagonal2011}.

In \cite{hartwigDiagonalReductionAlgebra2022a}, we determined generators and relations of the diagonal reduction superalgebra $A$ of $\Fosp(1|2)=\Fosp(1|2, \C)$ related to the diagonal embedding of $\Fosp(1|2)$ into its Cartesian square. The present paper continues an investigation of this orthosymplectic reduction superalgebra, addressing two fundamental goals: 
    \begin{enumerate}[{\hspace{.25in}\rm {Goal} 1:}]
        \item Describe the (ghost) center of $A$. \label{Goal1}
        \item Classify finite-dimensional irreducible  $A$-representations. \label{Goal2}
    \end{enumerate}
     
     From their examination of a classical diagonal reduction algebra \cite{khoroshkinDiagonalReductionAlgebra2017}, Khoroshkin and Ogievetsky provide two distinct families of central elements of the $\Fgl(n)$-type diagonal reduction algebra, and they leave to conjecture a full characterization of the center in general. In this paper we present a full characterization of the ghost center (which includes the center) of the diagonal reduction algebra of $\Fosp(1|2)$. Furthermore, we classify all finite-dimensional irreducible $A$-representations up to isomorphism. Along the way, we construct analogues of Verma modules for $A$ and study their respective Shapovalov forms, including the computation of their determinants.
     We define these structures over a PID of ``dynamical scalars'', $R=\C[H][(H-n)^{-1}\mid n\in\Z]$, which is a localized enveloping algebra of the $1$-dimensional Cartan subalgebra of $\Fosp(1|2)$. This is a parallel approach to so called \emph{dynamical representations} defined by Etingof and Varchenko \cite{etingofDynamicalWeylGroups2002}. 

\subsection{Ghost center}

The ghost center of a Lie superalgebra $\Fg$ was introduced by Gorelik in \cite{gorelikGhostCentreLie2000} as the direct sum of invariants under the standard and twisted adjoint actions \cite{arnaudonCasimirGhost1997}, respectively, within the universal enveloping algebra $U(\Fg)$ of $\Fg$. Equivalently, the ghost center is the center plus \emph{anti-center} of $U(\Fg)$. 
The even part of the anti-center of an associative superalgebra $B$ is defined as the set of even elements that commute with even elements of $B$ and anti-commute with odd elements of $B$. The odd part of the anti-center of $B$ comprises odd elements of $B$ that commute with all of $B$.
The ghost center carries a similar significance in super representation theory of Lie superalgebras, especially when studying $B(0,n$) \cite{kacLieSuperalgebras1977}, as the center holds in the representation theory of Lie algebras. 
Since the reduction algebra $A$ is an associative superalgebra, and a subquotient of an extension of a universal enveloping algebra of a Lie superalgebra, it is natural to pursue Goal \ref{Goal1} as a route towards Goal \ref{Goal2}. 

\subsection{Main results} 
 Let $\Z_{2}$ stand for the two-element field $\Z/2\Z$. The standard ground field will be the field $\C$ of complex numbers. A summary of the main results and the section containing each follows: 

In Section \ref{sec:GhostCenter} we deliver positive news on the conclusion of Goal \ref{Goal1}. The section includes a proof that the ghost center of $A$ is a subalgebra of $\C[x, y]$. Formally:
\begin{Theorem}[{Ghost center of $A$}] \label{thm:MainTheorem1}
The ghost center of the diagonal reduction algebra of $\Fosp(1|2)$ is isomorphic to the the subalgebra of $\C[x,y]$ generated by $x^2+y^2$, $2xy$, and $x^2-y^2$. An isomorphism $\phi$ is explicitly given by
\begin{align*}
\phi\big(C^{(1)}\big) = 2xy,\qquad
\phi\big(C^{(2)}\big) = x^2+y^2,\qquad
\phi\big(Q^{(2)}\big) = x^2-y^2,
\end{align*}
where $C^{(1)}$ and $C^{(2)}$ (defined in \eqref{eq:CCnegdiam}-\eqref{eq:CCposdiam}) are, up to scaling, the respective images in $A$ of $C\otimes 1 - 1\otimes C$ and $C\otimes 1+1\otimes C$, $C$ being the Casimir of $\Fosp(1|2)$; and, $Q^{(2)}$ (defined in \eqref{eq:Qdiam}) is, up to scaling, the image in $A$ of $Q\otimes Q$, where $Q$ is the Scasimir of $\Fosp(1|2)$. Moreover if we equip $\C[x,y]$ with a $\Z_2\times\Z_2$-grading given by 
\[\C[x,y]_{(\bar 0,b)}=\{f(x,y)\in\C[x,y]\mid f(x,y)=(-1)^bf(y,x)\},\quad \C[x,y]_{(\bar 1,b)}=0,\quad b\in\Z_2,\]
$\phi$ is a graded isomorphism, with respect to the $\Z_2\times\Z_2$-grading on the ghost center (see Section \ref{sec:return-ghost-center}).
\end{Theorem}

In Section \ref{sec:FDreps} we show that finite-dimensional irreducible $A$-representations arise as quotients of Verma modules defined in Section \ref{ssec:Verma}. Theorem \ref{thm:MainTheorem2}, proved in Lemma \ref{lem:action-of-ghost-center} and Theorem \ref{thm:L-lambda-mu}, completes Goal \ref{Goal2}.
\begin{Theorem}[{Finite-dimensional irreducible $A$-representations}]\label{thm:MainTheorem2}
\text{}
\begin{enumerate}[{\rm (i)}]
    \item For each $n\in 2\Z_{\ge 0}+1$ and each pair $(\la,\mu)\in\C\times(\C\setminus\Z)$ satisfying \begin{equation}\label{eq:lambda-mu-n}
    \la^2=(\mu+n)^2,
    \end{equation}
    there is an irreducible representation $L(\la,\mu)$ of dimension $n$.
    \item Each finite-dimensional irreducible $A$-representation $V$ is odd-dimensional and isomorphic to $L(\la,\mu)$ for a unique pair $(\la,\mu)\in\C\times(\C\setminus\Z)$ satisfying \eqref{eq:lambda-mu-n} where $n=\dim V$.
    \item The action of the ghost center on $L(\la,\mu)$ is given by
    \begin{equation}
        C^{(1)} \mapsto 2\la\mu,\quad C^{(2)}\mapsto \mu^2+\lambda^2,\quad     Q^{(2)}\mapsto (\mu^2-\lambda^2)(-1)^{|\cdot|},
    \end{equation}
    where $(-1)^{|\cdot|}\in\End_\C\big(L(\la,\mu)\big)$ denotes the parity sign function defined on homogeneous vectors by $v\mapsto (-1)^{|v|}$. 
\end{enumerate}
\end{Theorem} 

In Section \ref{sec:TensorProducts}, we appeal to Mickelsson's application of reduction algebras to representation theory though in the super setting. Namely, we demonstrate that the representation theory of the reduction superalgebra $A$ (Goal \ref{Goal2}) yields the following theorem on tensor product representations of $\Fosp(1|2)$.

\begin{Theorem}[Polynomial tensor product decomposition]\label{thm:MainTheorem3} Fix the root vector basis $\{x_{-2\al}, x_{-\al}, h, x_{\al}, x_{2\al}\}$ of $\Fosp(1|2)$ with positive root $\alpha \equiv -1 \in (\C h)^{\ast}$; consequently, $\Fosp(1|2) = \Fn_{-} \oplus \C h \oplus \Fn_{+}$, for $\Fn_{\pm} = \C x_{\pm2\al} \oplus \C_{\pm\al}$. 
Let $\ell$ be a non-negative integer. Let $V(-\ell)$ be the unique irreducible representation of $\Fosp(1|2)$ of dimension $2\ell+1$ with highest weight vector $v_\ell$. Let $\C[x]=V(\tfrac{1}{2})$ be the polynomial representation of $\Fosp(1|2)$ with highest weight vector $1$.
Then we have the following explicit decomposition into irreducible $\Fosp(1|2)$-representations:
\begin{equation}\label{eq:polytensordecomposition}
 \C[x]\otimes V(-\ell) = \bigoplus_{j=0}^{2\ell} U(\Fn_-) \big( \tilde x_{-\al} - \frac{1}{H-1}X_{-\al}\tilde h-\frac{1}{H-1}X_{-\al}^2\tilde x_\al-\frac{2}{(H-2)(H-1)}X_{-\al}^3\tilde x_{2\al}\big)^j (1\otimes v_\ell),
\end{equation}
where $X_{-\al}=x_{-\al}\otimes 1 + 1\otimes x_{\al}$, and $\tilde x_{\beta}=x_{\beta}\otimes 1 - 1\otimes x_{\beta}$, for $\be\in\{\pm\al,\pm 2\al\}$.
\end{Theorem}

We begin with background and a technical lemma in Section \ref{sec:Preliminaries} and conclude with its proof in Appendix \ref{appendix:A}. 

\section*{Acknowledgments}

The first author was supported by Simons Foundation Collaboration Grant \#637600. The second author was supported by his wife, Diamond Emelda Williams.
 
\section{Preliminaries}\label{sec:Preliminaries}
\subsection{The diagonal reduction algebra of \texorpdfstring{$\Fosp(1|2)$}{osp(1|2)}}\label{sec:DRAosp}

In \cite{hartwigDiagonalReductionAlgebra2022a}, the diagonal reduction algebra $Z(\FG, \Fg; D)$ of the Lie superalgebra $\Fosp(1|2)$ is initially given as a quotient algebra isomoprhic to the superalgebra $A$, whose definitions we state below. 
Let $\FG$ be the Lie superalgebra $\Fosp(1|2)\times\Fosp(1|2)$ in which $\Fg$ is the reductive image of the Lie superalgebra $\Fosp(1|2)$ under the diagonal embedding $\delta$. We let $\Fp$ be the reductive complement of $\Fg$. The multiplicative set $D$ is the monoid generated by integer shifts of our choice Cartan element $H \in \Fg$; throughout, the Cartan subalgebra of $\Fg$ is $\FH = \C H$. 
There is a basis $\{x_{-2\al}, x_{-\al}, h, x_{\al}, x_{2\al}\}$ of $\Fosp(1|2)$ with
the given relations under the supercommutator $[\cdot,\cdot]$ (and with the usage of $\pm = - \mp$ as a dependent parallel within any single equation):
\begin{gather*}
    [h,x_{\pm k\alpha}] = \mp k x_{k\alpha}, \qquad [x_{-k\alpha}, x_{k\alpha}] = h, \enskip k\in\{1,2\},\\  
    [x_{\pm\alpha}, x_{\pm\alpha}] =  \mp2x_{\pm2\alpha}, \qquad
    [x_{\pm\alpha}, x_{\mp2\alpha}] = x_{\mp\alpha}, \qquad
    [x_{\pm2\alpha}, x_{\pm\alpha}] = 0.
\end{gather*}
The Lie superalgebra $\Fosp(1|2)$ also has a triangular decomposition
\[\Fosp(1|2) = \Fn_{-} \oplus \Fh \oplus \Fn_{+}, \quad \Fh = \C h, \quad \Fn_{\pm} = \C x_{\pm 2\al} \oplus \C x_{\pm \al}.\]

Continuing, set $R=D^{-1}U(\FH)=\C[H][(H-n)^{-1}\mid n\in\Z]$ to be the ring of dynamical scalars.
We let $X_{\be} = (x_{\be},x_{\be}) \in \Fg$ and identify it with $x_{\be} \otimes 1 + 1 \otimes x_{\be} \in U(\FG)$ as an element of  $R \otimes_{U(\FH)} U(\FG)$; likewise, $\tilde x_{\be} = (x_{\be},- x_{\be}) \in \Fp$ is identified with $x_{\be} \otimes 1 - 1 \otimes x_{\be}$, for all root vectors $x_{\be}$, roots $\be \in \{\pm 2\alpha,\pm\alpha\}$. We also put $H$ for $(h,h)$ and $\tilde h$ for $(h,-h)$ and for their natural identifications, $h \otimes 1 + 1 \otimes h$ and $h \otimes 1 - 1 \otimes h$, respectively. Names for the spaces 
\begin{align*}\FN_{\pm} &= \C X_{\pm 2 \alpha} \oplus \C X_{\pm \alpha}\\ 
\tilde\Fh &= \C \tilde{h}\\ 
\tilde\Fn_{\pm} &= \C \tilde x_{\pm 2 \alpha} \oplus \C \tilde x_{\pm \alpha}
\end{align*} will be useful below. 
Underlying the theory is the $\Fg$-module decomposition
\[\FG = \Fg \oplus \Fp = (\FN_{-} \oplus \FH \oplus \FN_{+}) \oplus  (\tilde\Fn_{-} \oplus \tilde\Fh \oplus \tilde\Fn_{+})\] with $\Fg \cong \Fp$, $X \leftrightarrow \tilde{x}$, as $\Fg$-modules.

The associative algebra $U \supset U(\FG)$ along with the left ideal $I$ generated by $\FN_{+}$ and its normalizer $N_{U}(I)$ are used to define $Z$:

{\centering
$\displaystyle
\begin{aligned}
    U &= R \otimes_{U(\FH)} U(\FG),\\ 
    I &= U \FN_{+},\\ 
    Z &= Z(\FG, \Fg; D) = N_{U}(I)/I
\end{aligned}$
\par}
\vspace{1em}
\vspace{1em}
The algebra $Z$ is called the \emph{diagonal reduction algebra of $\Fosp(1|2)$}; furthermore, the canonical projection of super vector spaces $U \to U/\textup{\Romanbar{II}}$, where ``double I'' is the subspace $\textup{\Romanbar{II}} =   U \FN_{+} + \FN_{-} U$,  induces an isomorphism of $Z$ with the algebra \[A=(U/\textup{\Romanbar{II}},\bDiamond),\] where $\bDiamond$ is an associative product on the double coset space $U/\textup{\Romanbar{II}}$ defined through the extremal projector \cite{tolstoyExtremalProjectionsReductive1985, tolstoyExtremalProjectorsContragredient2011,hartwigDiagonalReductionAlgebra2022a}.
Generators of the reduction algebra $A$ (as an $R$-ring) are decorated like so: $\bar{x}_\beta=\tilde x_\beta+\textup{\Romanbar{II}}$, $\bar h=\tilde h+\textup{\Romanbar{II}}$, to distinguish them from elements of $\Fosp(1|2)$.
Lastly, we will need a non-localized version of $I$ which we denote\\

{\centering
$\displaystyle
\begin{aligned}
    \check{I} = U(\FG)\FN_{+}.
\end{aligned}$
\par}
\vspace{1em}
For a more thorough account of superified spaces, their maps, and other non-classical notions, readers are pointed to the texts \cite{chengDualitiesRepresentationsLie2012, mussonLieSuperalgebrasEnveloping2012} or Section 2-b of \cite{brundanHeckeCliffordSuperalgebrasCrystals2002}.

\subsection{Presentation and PBW basis for \texorpdfstring{$Z$}{Z}}
In \cite{hartwigDiagonalReductionAlgebra2022a} the following two theorems were proved and are key to proving the main results of the subsequent sections. 

\begin{Theorem}\label{thm:presentation}
	$A$ is generated as an $R$-ring by $\{ \bar x_{-2\al}, \bar x_{-\al}, \bar h, \bar x_\al, \bar x_{2\al}\}$ subject to the following relations:
	\begin{subequations}\label{eq:thm-rels}
		\begin{align}
            \bar x_{k\al} \bDiamond f(H) &= f(H+k)\bDiamond \bar x_{k\al}, \qquad \forall k\in\{\pm 1, \pm 2\},\;\forall f(H)\in R,\label{eq:xf}\\
			  \bar h \bDiamond f(H) &= f(H) \bDiamond  \bar h,\qquad\forall f(H)\in R,\\
			  \bar x_{2\al}  \bDiamond  \bar x_{\al}
			&=\big(1-\frac{2}{H+1}\big) \bar x_\al \bDiamond   \bar x_{2\al}
			\label{eq:2alphaDalpha}\\
        \bar x_\al  \bDiamond  \bar x_\al &= \frac{2}{H}   \bar h  \bDiamond  \bar x_{2\al} \label{eq:alphaDalpha}\\
        \bar x_{-\al} \bDiamond   \bar x_{-\al} &= -\frac{2}{H-2}   \bar x_{-2\al} \bDiamond   \bar h \label{eq:-alphaD-alpha}\\
			  \bar x_{2\al}  \bDiamond  \bar h &=
			\Big(1-\frac{2}{H+1}\Big)  \bar h  \bDiamond  \bar x_{2\al}
			\label{eq:2alphaDh} \\
			  \bar x_{2\al} \bDiamond   \bar x_{-\al} &=
			\Big(1-\frac{2}{H(H-1)}\Big) 
			  \bar x_{-\al} \bDiamond   \bar x_{2\al} 
			+\frac{2}{H+1}   \bar h  \bDiamond  \bar x_\al
			\label{eq:2alphaD-alpha}\\
			  \bar x_{2\al} \bDiamond   \bar x_{-2\al}&=
			\Big(1+2\frac{H^3 + H^2 - 6H + 4}{(H-2)(H-1)H(H+1)(H+2)}\Big) 
			  \bar x_{-2\al}  \bDiamond  \bar x_{2\al}
			\nonumber\\
			&\quad-\frac{H^2-H-1}{(H-1)H(H+1)}
  \bar x_{-\al}  \bDiamond  \bar x_\al 
			+\frac{1}{H+1}  \bar h  \bDiamond  \bar h 
			+\frac{-H^2}{H+1} 
			\label{eq:2alphaD-2alpha}\\
			  \bar x_\al \bDiamond    \bar h &=
			\big(1-\frac{1}{H}\big)  \bar h  \bDiamond  \bar x_\al  
			\label{eq:alphaDh} \\
			  \bar x_\al  \bDiamond   \bar x_{-\al} &=
            \Big(-1+\frac{-1}{H-1}\Big) 
			  \bar x_{-\al} \bDiamond   \bar x_\al \nonumber \\
			&\quad + \frac{4H}{(H-1)(H-2)}  \bar x_{-2\al} \bDiamond \bar x_{2\al} -\frac{1}{H}   \bar h \bDiamond \bar h 
			+H
			\label{eq:alphaD-alpha} \\
			  \bar x_\al  \bDiamond   \bar x_{-2\al} & = 
      \Big(1-\frac{2}{(H-1)(H-2)}\Big)  \bar x_{-2\al} \bDiamond  \bar x_\al 
			-\frac{2}{H} 
			  \bar x_{-\al} \bDiamond   \bar h
			\label{eq:alphaD-2alpha}\\
			  \bar h  \bDiamond   \bar x_{-\al} &=
			\big(1-\frac{1}{H-1}\big) \bar x_{-\al}  \bDiamond  \bar h 
			\label{eq:hD-alpha}\\
			  \bar h  \bDiamond   \bar x_{-2\al} &=
			\Big(1-\frac{2}{H-1}\Big) \bar x_{-2\al} \bDiamond   \bar h
			\label{eq:hD-2alpha}\\
			  \bar x_{-\al}  \bDiamond   \bar x_{-2\al} &=\big(1-\frac{2}{H-2}\big) \bar x_{-2\al} \bDiamond    \bar x_{-\al}
			\label{eq:-alphaD-2alpha}
		\end{align}
	\end{subequations}
\end{Theorem}

\begin{Theorem} \label{thm:PBW}
	$A$ is a free left (and right) $R$-module on the following set of monomials:
	\begin{equation}\label{PBWforAsetofmonomials}
		\{\bar x_{-2\al}^{\bDiamond p}\bDiamond\bar x_{-\al}^{\bDiamond q}\bDiamond\bar h^{\bDiamond r}\bDiamond\bar x_\al^{\bDiamond s}\bDiamond \bar x_{2\al}^{\bDiamond t}
		\mid p,q,r,s,t\in\Z_{\ge 0},\, q,s\le 1\}.
	\end{equation}
\end{Theorem}
 The $\Z_{2}$-grading on $A$ is the induced grading from the construction beginning with $\Fosp(1|2)$. From Theorem \ref{thm:PBW}, specifically, $A = A_{\overline{0}} \oplus A_{\overline{1}}$ with even part comprising sums of monomials having  $q = s$ and odd part comprising sums of monomials having $q \neq s$.

\subsection{Normalized generators of \texorpdfstring{$A$}{A} in the Mickelsson algebra}

The Mickelsson algebra \cite{mickelssonStepAlgebrasSemisimple1973} is a non-localized version of the reduction algebra. For a reductive pair of Lie superalgebras $\Fl\subseteq \Fk$, fixing a triangular decomposition of $\Fl$ with $\Fl_{+}$ a summand, it is defined as $N_{U(\Fk)}(U(\Fk)\Fl_{+})/\U(\Fk)\Fl_{+}$.
The Mickelsson algebra may be regarded as a subalgebra of the corresponding reduction algebra. Generally, the Mickelsson algebra is not finitely generated as a $\C$-algebra \cite{herlemontDifferentialCalculusMathbfh2018}. One purpose of localization is to define the corresponding reduction algebra as a finitely-generated $R$-ring. 

In this subsection we provide elements of the Mickelsson algebra which correspond to the generators of $A$. This means we are in the setting of the Mickelsson algebra for the reductive pair $\Fg \subset \FG$ when $\FG=\Fosp(1|2)\times\Fosp(1|2)$ and $\Fg$ is the image of $\Fosp(1|2)$ under the diagonal embedding. The following argument \cite[see][Section 2.1]{khoroshkinContravariantFormReduction2018} shows that the Mickelsson algebra can be considered a subalgebra  of the reduction algebra. From the containment $\check{I} \subset I$ through the embedding $U(\FG) \subset U$ (see Section \ref{sec:DRAosp}):
Each element $u$ of $N_{U(\FG)}(\check{I})$ satisfies $\FN_{+} u\subset \check{I} \subset I$, yielding $I u = U\FN_{+} u \subset UI \subset I$.
Thus there is an inclusion $N_{U(\FG)}(\check{I})\to N_U(I)$.
Moreover, $\check{I} \subset I$ implies that $\check{I}$ is contained in the kernel of $N_{U(\FG)}(\check{I})\to N_U(I)/I$. In fact, by the PBW Theorem for $\FG$ with respect to an ordered basis of the form $(H,\ldots,X_{\al},X_{2\al})$, the ideal $\check{I}$ is equal to the kernel of $N_{U(\FG)}(\check{I})\to N_U(I)/I$.
Thus there is an injective homomorphism of associative 
superalgebras
\[
N_{U(\FG)}(\check{I})/\check{I} \to N_U(I)/I \cong A.
\]

Define the following elements of the reduction algebra $A$:
\begin{equation}\label{eq:hat-variables}
\left\{
\begin{aligned}
\hat x_{2\al} &= \bar x_{2\al}\\
\hat x_\al &= (H-1)\bar x_{\al}\\
\hat h &= (H-1)\bar h\\
\hat x_{-\al} &= (H-1)(H-2)\bar x_{-\al}\\
\hat x_{-2\al} &= (H-1)(H-2)\bar x_{-2\al}
\end{aligned}
\right.
\end{equation}
From the explicit formulas \cite[Section 3]{hartwigDiagonalReductionAlgebra2022a} for the elements $\bar x_\beta$, $\bar h$, multiplying by $(H-1)$ and $(H-1)(H-2)$ in this way has the effect of canceling denominators. Thus $\hat x_\beta$ and $\hat h$ all belong to the image of the Mickelsson algebra in the algebra $A$, in a reasonable sense. We cite Zhelobenko \cite[Eq. (4.16)]{zhelobenko1985gelfand} for this method  of producing elements of the Mickelsson algebra from the reduction algebra. Here we record the relations for these normalized generators (of $A$), which we use in the upcoming lemma and throughout the following sections.

\begin{Theorem}\label{thm:nonlocalgen}
The Mickelsson algebra elements described above obey the following relations: 
\begin{subequations} 
\begin{align}
\label{eq:relH}
\hat x_{k\al} \bDiamond H &= (H+k) \bDiamond  \hat x_{k\al},\quad k\in\{\pm 1,\pm 2\},\\
\label{eq:relh}
\hat h &\quad\text{ is central,}\\
\label{eq:rel11}
\hat x_\al \bDiamond  \hat x_\al &= 2\hat h \bDiamond \hat x_{2\al}\\
\label{eq:rel-1-1}
\hat x_{-\al} \bDiamond \hat x_{-\al} &= -2\hat h \bDiamond \hat x_{-2\al},\\
\label{eq:rel12}
\hat x_\be \bDiamond \hat x_{2\be} &= \hat x_{2\be} \bDiamond \hat x_\be, \quad \be\in\{\al,-\al\},\\
\label{eq:rel2-1}
(H-1)^2\hat x_{2\al} \bDiamond \hat x_{-\al} &= (H+1)^2\hat x_{-\al} \bDiamond \hat x_{2\al} + 2H\hat h \bDiamond \hat x_\al,\\
\label{eq:rel1-1}
(H-2)^2\hat x_{\al} \bDiamond \hat x_{-\al} &= -H^2 \hat x_{-\al} \bDiamond \hat x_{\al} +4H^2 \hat x_{-2\al} \bDiamond \hat x_{2\al} \nonumber \\&\quad -(H-2)^2\hat h\bDiamond\hat h + H^2(H-1)^2(H-2)^2,\\
\label{eq:rel1-2}
(H-2)^2\hat x_\al \bDiamond \hat x_{-2\al} &= H^2 \hat x_{-2\al} \bDiamond \hat x_\al - 2(H-1)\hat x_{-\al} \bDiamond \hat h,\\
\label{eq:rel2-2}
(H-1)^2(H-2)^2\hat x_{2\al} \bDiamond \hat x_{-2\al} &= H^2(H-1)^2\hat x_{-2\al} \bDiamond \hat x_{2\al}+(-H^2+H+1)\hat x_{-\al} \bDiamond \hat x_\al \nonumber \\&\quad +H(H-2)^2\hat h\bDiamond \hat h-H^3(H-1)^2(H-2)^2 
\end{align}
\end{subequations}
\end{Theorem}

\begin{proof}
The calculations follow from substitution: Use \eqref{eq:hat-variables} and  Theorem \ref{thm:presentation}.
\end{proof}

\subsection{A technical lemma}
We will use $\hat h$ defined in \eqref{eq:hat-variables} by
\begin{equation}\label{eq:hath}
\hat h=(H-1)h.
\end{equation}
It was proved in \cite{hartwigDiagonalReductionAlgebra2022a} (and in more detail below; see Lemma \ref{centralElementsA}) that $\hat h$ is a central element of $A$. We will frequently use $\hat h$ as a generator in place of $h$. Note that $R[\hat h]=R[\bar h]$, although expressing a given element in either $R$-ring will result in different coefficients (for example, \eqref{eq:hath}).

\begin{Lemma}\label{lem:tech}
\text{}
\begin{enumerate}[{\rm (i)}]
\item \label{lem:techi} For any non-negative integer $n$,
\begin{equation}\label{eq:Fn-evenformula}\hat x_{2\al}\hat x_{-2\al}^n \equiv -nH^2(H-n+1) \hat x_{-2\al}^{n-1} \quad\mod{A\hat x_{\al}+A\hat x_{2\al}+A\hat h}.
\end{equation}
\item For any non-negative integer $n$,
\begin{equation}\label{eq:Fn-congruencebar}
\bar x_\al \bar x_{-\al}^n \equiv
 F_n(H, \hat h) \bar x_{-\al}^{n-1} \mod{A\bar x_\al+A\bar x_{2\al}},
\end{equation}
where
\begin{equation}
 H(H-1)^2 
 F_n(H,\hat h)=
\begin{cases}
\left(H^2(H-n)^{-2}-1\right) \hat h^2,& \text{$n$ even,}\\
 H^2(H-n)^2-\hat h^2,&\text{$n$ odd}.
\end{cases}
\end{equation}
\item For any non-negative integer $n$,
\begin{equation}\label{eq:Fn-congruencehat}
\hat x_\al \hat x_{-\al}^n \equiv
\hat F_n(H, \hat h) \hat x_{-\al}^{n-1} \mod{A\hat x_\al+A\hat x_{2\al}},
\end{equation}
where
\begin{equation} \label{eq:Fn-formula}
\hat F_n(H,\hat h)=
\begin{cases}
\left(H^2(H-n)^{-2}-1\right) \hat h^2,& \text{$n$ even,}\\
 H^2(H-n)^2-\hat h^2,&\text{$n$ odd}.
\end{cases}
\end{equation}
\end{enumerate}
\end{Lemma}

\begin{proof}
(i) Consider the relations given in Theorem \ref{thm:nonlocalgen}, in particular, \eqref{eq:rel2-2}. A proof by induction on $n$ gives the result.

(ii) The proof for this congruence is more involved; see Appendix \ref{appendix:A} for a complete proof.

(iii) Part (ii) and substitution for the normalized generators defined in \eqref{eq:hat-variables} imply the result.
\end{proof}
 
 With the tool of Lemma \ref{lem:tech} in play, we now proceed in attending to Goals \ref{Goal1} and \ref{Goal2} throughout Sections \ref{sec:GhostCenter} and \ref{sec:FDreps}.

\section{The ghost center of the diagonal reduction algebra of \texorpdfstring{$\Fosp(1|2)$}{osp(1|2)}}\label{sec:GhostCenter}

It is well-documented \cite[as a seminal example]{kacLieSuperalgebras1977} that Casimir operators play a key role in determining both central structure and representation theory of the universal enveloping algebra of a Lie superalgebra \cite{pinczonEnvelopingAlgebraLie1990}. Moreover, \cite{arnaudonCasimirGhost1997} explains the existence of the so-called Scasimir introduced in \cite{lesniewskiRemarkCasimirElements1995}. We show that the diagonal reduction algebra $A$ of $\Fosp(1|2)$ also has Casimir and Scasimir elements. 

\subsection{Casimir and Scasimir}
To describe the (S)Casimir elements of $A$, we consider the image in $A$ of three elements of the universal enveloping algebra $U(\FG)$ while making use of the natural isomorphism $U(\FG) \cong U(\Fosp(1|2)) \otimes U(\Fosp(1|2))$.

In \cite{lesniewskiRemarkCasimirElements1995}, a quadratic Casimir $C \in U(\Fosp(1|2))$ is written in $\Fosp(1|2)$ bosonic generators $L_\pm = -x_{\mp 2\al}, L_{3} = \frac{1}{2}h$ and fermionic generators $G_\pm = -\frac{\sqrt{-1}}{2} x_{\mp\al}$ such that (see \cite{hartwigDiagonalReductionAlgebra2022a} for conversion between sets of generators)
\begin{align}C &=   
 \frac{1}{2}(L_{+}L_{-} + L_{-}L_{+}) + L_{3}^{2} + G_{+}G_{-}-G_{-}G_{+} + \frac{1}{16}\notag\\ 
 &= \frac{1}{2}(x_{-2\alpha}x_{2\alpha} + x_{2\alpha}x_{-2\alpha}) + \frac{1}{4}h^{2} + \frac{-1}{4}(x_{-\alpha}x_{\alpha} - x_{\alpha}x_{-\alpha}) + \frac{1}{16}\notag\\
 &= x_{-2\alpha}x_{2\alpha} + \frac{-1}{2}x_{-\alpha}x_{\alpha} + \frac{1}{4}(h^{2}-h) + \frac{1}{16}.
\end{align}

Now consider two special elements \[\CC_{\pm}= C \otimes 1 \pm 1 \otimes C + I \in N_{U}(I)/I,\]
as in \cite[Section 6.5]{khoroshkinStructureConstantsDiagonal2011}. Note that membership holds
since $C \, \otimes \,1$ and $1 \, \otimes \,C$ are in $N_{U}(I)$ as elements of the center of $U(\Fosp(1|2)) \otimes U(\Fosp(1|2))$.

The following lemma expresses the projection of $\CC_{\pm}$ in $A$ in the PBW basis \eqref{PBWforAsetofmonomials}.
\begin{Lemma}\label{centralElementsA}
The elements $\bar\CC_{\pm} = C\otimes 1\pm 1\otimes C+ \emph{\Romanbar{II}}$ are central in $A$, with 
\begin{align}
    C^{(1)} = 8\bar{\CC}_{-} & = 2(H-1)\bDiamond\bar{h} \label{eq:CCnegdiam}\\
    C^{(2)} = 8\bar{\CC}_{+} & = \left(4 + \frac{4}{H-2} \right)\bar{x}_{-2\alpha} \bDiamond \bar{x}_{2\alpha} - \left(2 - \frac{1}{H-1}\right)\bar{x}_{-\alpha} \bDiamond \bar{x}_{\alpha} + \bar{h} \bDiamond \bar{h} + (H-1)^{2} \label{eq:CCposdiam}
\end{align}
\end{Lemma}
\begin{proof}
Recalling $A \stackrel{\sim}{\leftarrow} N_{U}(I)/I$, explicitly, $n + \textup{\Romanbar{II}} \mapsfrom n + I$, for any normalizing element $n$, then the centrality of $\bar{\CC}_{\pm}$ follows from construction. 
Now taking computations modulo $\textup{\Romanbar{II}}$, we have
\begin{align*} 
C \otimes 1 \phantom{~} & \overset{\textup{\Romanbar{II}}}{\equiv} \phantom{~} \frac{1}{4}\left(X_{-2\alpha} + \tilde{x}_{-2\alpha} \right)\left(X_{2\alpha} + \tilde{x}_{2\alpha} \right)-\frac{1}{8}\left(X_{-\alpha} + \tilde{x}_{-\alpha}\right)\left(X_{\alpha} + \tilde{x}_{\alpha}\right) \notag\\
& \enskip + \frac{1}{16}\left(H + \tilde{h}\right)\left(H + \tilde{h}\right) - \frac{1}{8}\left(H + \tilde{h}\right) + \frac{1}{16} \notag\\
\phantom{~} &\overset{\textup{\Romanbar{II}}}{\equiv} \phantom{~}
\frac{1}{4}\left(X_{-2\alpha}X_{2\alpha} + X_{-2\alpha}\tilde{x}_{2\alpha} + \tilde{x}_{-2\alpha}X_{2\alpha} + \tilde{x}_{-2\alpha}\tilde{x}_{2\alpha}\right) -\frac{1}{8}\left(X_{-\alpha}X_{\alpha} + X_{-\alpha}\tilde{x}_{\alpha} + \tilde{x}_{-\alpha}X_{\alpha} + \tilde{x}_{-\alpha}\tilde{x}_{\alpha}\right)\notag \\
& \enskip + \frac{1}{16}\left(H^{2} + 2H\tilde{h} + \tilde{h}^{2}\right) - \frac{1}{8}\left(H + \tilde{h}\right) + \frac{1}{16} \notag\\
\phantom{~} &\overset{\textup{\Romanbar{II}}}{\equiv} \phantom{~}
\frac{1}{4} \tilde{x}_{-2\alpha}\tilde{x}_{2\alpha} -\frac{1}{8} \tilde{x}_{-\alpha}\tilde{x}_{\alpha}
+ \frac{1}{16}\left(H^{2} + 2H\tilde{h} + \tilde{h}^{2}\right) - \frac{1}{8}\left(H + \tilde{h}\right) + \frac{1}{16}
\end{align*}
and
\begin{align*} 
1 \otimes C \phantom{~} &\overset{\textup{\Romanbar{II}}}{\equiv} \phantom{~} \frac{1}{4}\left(X_{-2\alpha} - \tilde{x}_{-2\alpha} \right)\left(X_{2\alpha} - \tilde{x}_{2\alpha} \right)-\frac{1}{8}\left(X_{-\alpha} - \tilde{x}_{-\alpha} \right)\left(X_{\alpha} - \tilde{x}_{\alpha} \right)\\ \notag 
& \enskip + \frac{1}{16}\left(H - \tilde{h}\right)\left(H - \tilde{h}\right) - \frac{1}{8}\left(H - \tilde{h}\right) + \frac{1}{16} \notag\\
\phantom{~} &\overset{\textup{\Romanbar{II}}}{\equiv} \phantom{~}
\frac{1}{4}\left(X_{-2\alpha}X_{2\alpha} - X_{-2\alpha}\tilde{x}_{2\alpha} - \tilde{x}_{-2\alpha}X_{2\alpha} + \tilde{x}_{-2\alpha}\tilde{x}_{2\alpha}\right) -\frac{1}{8}\left(X_{-\alpha}X_{\alpha} - X_{-\alpha}\tilde{x}_{\alpha} - \tilde{x}_{-\alpha}X_{\alpha} + \tilde{x}_{-\alpha}\tilde{x}_{\alpha}\right)\notag \\
& \enskip + \frac{1}{16}\left(H^{2} - 2H\tilde{h} + \tilde{h}^{2}\right) - \frac{1}{8}\left(H - \tilde{h}\right) + \frac{1}{16} \notag\\
\phantom{~} &\overset{\textup{\Romanbar{II}}}{\equiv} \phantom{~}
\frac{1}{4} \tilde{x}_{-2\alpha}\tilde{x}_{2\alpha} -\frac{1}{8} \tilde{x}_{-\alpha}\tilde{x}_{\alpha}
+ \frac{1}{16}\left(H^{2} - 2H\tilde{h} + \tilde{h}^{2}\right) - \frac{1}{8}\left(H - \tilde{h}\right) + \frac{1}{16}.
\end{align*}
Therefore,
\begin{equation}\label{eq:CCnegtilde}
\bar{\CC}_{-} = \frac{1}{4}(H-1)\tilde{h} + \textup{\Romanbar{II}},\end{equation} and
\begin{equation}\label{eq:CCpostilde}
\bar{\CC}_{+} = \frac{1}{2} \tilde{x}_{-2\alpha}\tilde{x}_{2\alpha} -\frac{1}{4} \tilde{x}_{-\alpha}\tilde{x}_{\alpha}
+ \frac{1}{8}\left(H^{2} + \tilde{h}^{2}\right) - \frac{1}{4}H + \frac{1}{8} + \textup{\Romanbar{II}}.
\end{equation}
Applying \cite[Lemma 3.2]{hartwigDiagonalReductionAlgebra2022a}
and multiplying by $8$ yields \eqref{eq:CCnegtilde} $\rightarrow$ \eqref{eq:CCnegdiam} and \eqref{eq:CCpostilde} $\rightarrow$ \eqref{eq:CCposdiam}.
\end{proof}

From here on, let
\begin{equation}
Q^{(2)}  = 16\bar{\CL} =  4\frac{H-2}{H-1} \bar x_{-2\al} \bDiamond \bar x_{2\al} - \left(2(H-2) + \frac{1}{H-1}\right) \bar x_{-\al} \bDiamond \bar x_{\al} - \bar{h}\bDiamond\bar{h} + (H-1)^{2}. \label{eq:Qdiam}
\end{equation}
It has already been shown \cite{hartwigDiagonalReductionAlgebra2022a} that 
$Q^{(2)}$ (in \emph{loc. cit.} denoted $16\bar\CL$)
is the Scasimir analogue in the diagonal reduction algebra $A$ for $\Fosp(1|2)$. In other words, $Q^{(2)}$ is an even anti-central element whose square is a Casimir operator. Moreover, $Q^{(2)}$
is unique up to a nonzero complex scalar multiple.

We will revisit \eqref{eq:CCnegdiam}, \eqref{eq:CCposdiam}, and \eqref{eq:Qdiam}, after defining a Harish-Chandra homomorphism, to provide generators of the ghost center described below. 

\begin{Remark}
The next couple of sections will feature unadorned elements of $A$ except for the normalized generators of the Mickelsson algebra from Theorem \ref{thm:nonlocalgen}. In particular, the generators of $A$ in Theorem \ref{thm:presentation} will be written as $x_{-2\al},\, x_{-\al},\, h,\, x_\al,\, x_{2\al}$. Likewise, we will write products in $A$ as $xy$ instead of $x \bDiamond y$ from here on. Basically, we drop bars and rock hats like the old Jay-Z. \cite[2nd line of 2nd verse]{jay-zfeat.aliciakeysEmpireStateMind2009}
\end{Remark}

\subsection{Harish-Chandra homomorphism}
Let $B$ be the $R$-subring of $A$ generated by $\{x_\al,x_{2\al},h\}$.
By the PBW theorem (Theorem \ref{thm:PBW}), $B$ has a left $R$-basis given by monomials $x_{\al}^px_{2\al}^qh^r$ where $p,q,r\in\Z_{\ge 0}, p\le 1$.

Let $A^H$ denote the centralizer of $H$ in $A$. By the PBW theorem (Theorem \ref{thm:PBW}), $A^H$ is a free left $R$-module with basis given by monomials
\begin{equation}\label{eq:centralizer-monomials}
x_{-2\al}^px_{-\al}^q h^r x_\al^q x_{2\al}^p,\qquad q\in\{0,1\}, p,r\in\Z_{\ge 0}.
\end{equation}
Note that this implies that $A^H$ is contained in the even subalgebra $A_{\bar 0}$ of $A$.
Consider the left ideal $J$ of $A^H$ defined by
\begin{equation}
J=(Ax_\al+Ax_{2\al})\cap A^H.
\end{equation}
An arbitrary element of $J$ is a linear combination of monomials
\eqref{eq:centralizer-monomials} with $\min\{p,q\}>0$. This shows that
\begin{equation}
J=(x_{-\al}A+x_{-2\al}A)\cap A^H.
\end{equation}
Consequently $J$ is a two-sided ideal of the centralizer $A^H$ and, moreover
\begin{equation}\label{eq:HcentralizerDecomp}
A^H = R[h]\oplus J.
\end{equation}

\begin{Definition}
The \emph{Harish-Chandra homomorphism} \begin{equation}\label{eq:HC}
\varphi:A^H\to R[h]
\end{equation}
is the projection onto $R[h]$ with kernel $J$.
\end{Definition}
Since $J$ is a two-sided ideal, $\varphi$ is a $\C$-algebra homomorhpism along with being a map of $R$-rings.
\begin{Example}
$\varphi(C^{(1)}) = 2(H-1)h$ and $\varphi(Q^{(2)})=(H-1)^2-h^2$ by \eqref{eq:CCnegdiam} and \eqref{eq:Qdiam}, respectively.
\end{Example}

\subsection{Verma modules}\label{ssec:Verma}
\begin{Definition}\label{def:parityreversal}
We use $V\Pi$ to denote the \emph{right parity reversing functor} \cite[Section 4.1]{landweberRepresentationRingsLie2006} on a super module V. 
For an $(A,R)$-bimodule $V$, there is a natural action on $V\Pi$: We can view $V\Pi$ as the $(A,R)$-bimodule $V\otimes_R R^{0|1}$ under $a.(v\otimes r) = (av) \otimes r$. Duly noted, $R$ here is considered a purely even ground ring/algebra. The modules $V$ and $V\Pi$ appear below in the special case of Verma modules.
\end{Definition}
For $\la\in R$, let $R_\la$ be the purely even space $R$ equipped with the $(B,R)$-bimodule structure
$x_\al.r = x_{2\al}.r = 0$, $h.r = \la r$, $f.r=fr$ for all $r\in R$, $f\in R$. Here we define the Verma module $M(\la,\overline{0})$ to be the induced module
$M(\la,\overline{0}) = A\otimes_B R_\la$. At the risk of protraction, we state the action of $A$ on decomposable tensors of  $M(\la,\overline{0})$ by $a.(x\otimes r) = ax \otimes r$, for $a, x \in A$, $r \in R_{\lambda}$. 
Now we appeal to the PBW theorem and Theorem \ref{thm:presentation} and the diagonal action of $H$ and $h$ on the highest weight vector (see Section \ref{sec:TensorProducts} for a complementary discussion, including how generators of $N_{U}(I)/I$ describe the action of $h$ on $A$-modules): The right $B$-module $A$ is free on $\big\{x_{-2\alpha}^{p}, x_{-\alpha}^{q} \mid p \in \Z_{\geq 0}, q \in \{0, 1\}  \big\}$; moreover, $ax \otimes r$ can be expressed as $a' \otimes 1_{R}$, where $a'$ is in $A_{-}$,  the $R$-subring of $A$ generated by $\{x_{-2\alpha}, x_{-\alpha} \}$.

Let $v_\la = 1_A\otimes 1_R \in M(\la,\overline{0})$ be the (even) highest weight vector and $v_{\lambda,p,q} = x_{-2\alpha}^{p} x_{-\alpha}^{q}.v_{\lambda}$ so that 
\begin{equation}\label{eq:Mlambda-basis}
M(\lambda,\overline{0}) = \bigoplus_{\substack{p\in\Z_{\geq 0} \\ q \in \{0,1\}}} R v_{\lambda,p,q} ,\qquad   |v_{\lambda,p,q}| = \bar{q} \in \Z_{2}.
\end{equation}
When $\lambda\in R^\times$, the central element $\hat h$ acts on $M(\la,\overline{0})$ by multiplication of an invertible element of $R$.
Thus, substituting $h=\frac{1}{H-1}\hat h$ into \eqref{eq:-alphaD-alpha},
we may solve for $x_{-2\alpha}$ in terms of $x_{-\alpha}$ and $\hat h^{-1}$ which yields
\begin{equation}\label{eq:Mlambdanot0-basis}
M(\lambda,\overline{0}) = \underset{q \in \Z_{\geq 0}}{\bigoplus} R
v_{\lambda,0,q} ,\qquad   |v_{\lambda,0,q}| = \bar{q} \in \Z_{2}.
\end{equation}
By Definition \ref{def:parityreversal} we have 
$M(\la,\overline{0})\Pi= (A\otimes_B R_\la) \otimes_R R^{0|1}=A\otimes_B (\Pi R_\la)$. In this module, which we call $M(\la,\overline{1})$, the highest weight vector $v_{\lambda}$ will be odd.
Note that $M(\la,\overline{0}) \cong M(\la,\overline{1})$ as $A$-modules \cite{gorelikMinimalPrimitiveSpectrum2000} (if we allow odd intertwining operators). Therefore, from now on we choose to focus on 
\[
M(\lambda) = M(\la,\overline{0})
\]
as ``the'' universal highest weight $A$-module of weight $\lambda$.

We stress that the right tensor factor of $M(\lambda)$, namely $R_\la$, is naturally a $(B,R)$-bimodule and thus the Verma module is an $(A,R)$-bimodule. We will use the right $R$-module structure on $M(\la)$ below.

\subsection{Return of the ghost center}\label{sec:return-ghost-center}
We suggest to keep in mind that $A$ is both a $\C$-superalgebra and a finitely-generated $R$-superring.
\begin{Definition} 
Gorelik \cite{gorelikGhostCentreLie2000} defined the ghost center of the enveloping algebra of a Lie superalgebra. The same definition (focusing on the product instead of adjoints) works for associative superalgebras and superrings.
The \emph{center} of an associative superalgebra $B$ is $\{a\in B\mid [a,b]=0,\;\forall b\in B\}$ where $[\cdot,\cdot]$ denotes the super commutator.
The even part $\antiZ_{\bar{0}}(B)$ of the \emph{anti-center} of $B$ consists of the even elements of $B$ which anti-commute with odd elements and commute with even elements inside $B$. The odd part $\antiZ_{\bar{1}}(B)$ of the anti-center of $B$ consists of the odd elements of $B$ which commute with all of $B$.
The \emph{ghost center} is the sum $Z \oplus \antiZ$ of the center and the anti-center. We let $\ghostZ(B)$ denote the ghost center of $B$.
\end{Definition}
Note that the ghost center is a $(\Z_{2})\times(\Z_{2})$-graded algebra with $\ghostZ_{(i,\bar 0)}=Z_i$ and $\ghostZ_{(i,\bar 1)}=\antiZ_i$. If $z\in\ghostZ_{(i,j)}$, then $i$ is the usual parity of $z$ denoted $|z|$, while $j$ is a new parity we call \emph{ghost parity} and is denoted by $\{z\}$. Thus $\{z\}=\bar{0}$ if $z$ is central, and $\{z\}=\bar{1}$ if $z$ is anti-central. For a given associative superalgebra $B$ under consideration as an associative (non-super) algebra, the center is the sum of elements whose standard parity equals their ghost parity. That is, $Z_{\catname{Alg}}(B) = \ghostZ(B)_{(\bar{0}, \bar{0})} \oplus \ghostZ(B)_{(\bar{1},\bar{1})}$, where $Z_{\catname{Alg}}$ identifies the center of an object in the category of $\C$-algebras.
The ghost center of $A$ is not an $R$-subring: While $\hat h$ is in the ghost center (being central), $H\hat h$ is not, since $x_{\alpha}H\hat h - H\hat{h}x_{\alpha} = (H+1)\hat{h}x_{\alpha} - H\hat{h}x_{\alpha} =\hat h x_{\alpha}$ in $A$.

We observe that any element of the ghost center of our diagonal reduction algebra $A$ has to commute with the even element $H$ and thus belongs to the centralizer $A^H$. In particular, by \eqref{eq:centralizer-monomials}, the ghost center is a purely even subalgebra of $A$. That is, $\ghostZ(A)=\ghostZ(A)_{(\bar 0,\bar 0)}\oplus\ghostZ(A)_{(\bar 0,\bar 1)}$.

Now for a moment consider only the associative algebra structure of $A$. The two preceding paragraphs give a description of $Z_{\catname{Alg}}(A)$ analogous to \cite[Remark 3.5.1]{gorelikGhostCentreLie2000}: $Z_{\catname{Alg}}(A)= \ghostZ(A)_{\bar{0},\bar{0}}$.

The introduction of the ghost center of $A$ aligns with our strategy to fulfill both Goals \ref{Goal1} and \ref{Goal2}.
\begin{Lemma}\label{lem:AmoduleAction}
\text{}
\begin{enumerate}[{\rm (i)}]
\item If $v$ is any vector in an $A$-module such that $x_\al .v= x_{2\al} .v = 0$, then
\begin{equation}\label{eq:z-on-singular}
a.v = \varphi(a).v,\quad \forall a\in A^H. 
\end{equation}
\item Let $z$ be either a central or anti-central element of $A$. Let $\la\in R$ be any dynamical scalar. Put $z_0(H,\hat h)=\varphi(z)\in R[\hat h]$.

Then we have
\begin{equation}\label{eq:z-vs-HC}
z.u = (-1)^{\{z\} |u|}u.z_0(H,\hat\la) \quad \forall u\in M(\la).
\end{equation}
where $\hat\la=(H-1)\la$.
\end{enumerate}
\end{Lemma}

\begin{proof}
(i) Noting \eqref{eq:HcentralizerDecomp}, $a \in R[h] \oplus J$. Assumptions on $v$ imply $J.v = 0$. Thus the result now follows from the definition of the Harish-Chandra homomorphism $\varphi$. 

(ii) It suffices to show the identity when $u=x_{-2\al}^px_{-\al}^q .v_\la$ for some $p\in\Z_{\ge 0}, q\in\{0,1\}$. Then we have
\begin{align*}
z.u &= (-1)^{\{z\} |u|} x_{-2\al}^px_{-\al}^q z.v_\la\\
&\overset{(i)}{=} 
(-1)^{\{z\} |u|} x_{-2\al}^px_{-\al}^q z_0(H,\hat h) .v_\la\\
&=
(-1)^{\{z\} |u|} x_{-2\al}^px_{-\al}^q z_0(H,\hat\la) .v_\la \qquad\text{since $\hat h .v_\la=\hat\la .v_\la$,}\\
&=(-1)^{\{z\} |u|} u. z_0(H,\hat \la).
\end{align*}
\end{proof}

\subsection{Injectivity of the restriction of the Harish-Chandra homomorphism to the ghost center}

\begin{Proposition}
The restriction of the Harish-Chandra homomorphism $\varphi:A^H\to R[\hat h]$ to the ghost center of $A$ is injective.
\end{Proposition}

\begin{proof}
Let $z$ be an element of the ghost center such that $\varphi(z)=0$. Suppose $z\neq 0$ and write
\[z=\sum_{n=N}^\infty C_n(H,\hat h) x_{-2\al}^px_{-\al}^q x_\al^q x_{2\al}^p,\]
where $C_n(H,\hat h)\in R[\hat h]$ (at most finitely many nonzero) and $C_N\neq 0$, $N>0$, and we write $n=2p+q$, $p\in\Z_{\ge 0}$, $q\in\{0,1\}$. Consider an arbitrary Verma module $M(\la)$, $\la\in R$. Let $n=N=2p+q$ and consider the vector
\[v_N=x_{-2\al}^px_{-\al}^q\otimes 1\in M(\la).\]
Then we have
\[
z.v_N = C_N(H,\hat h)x_{-2\al}^px_{-\al}^q x_\al^qx_{2\al}^p .(x_{-2\al}^px_{-\al}^q\otimes 1),
\]
since $x_{2\al}^rx_\al^s.v_N=0$ if $2r+s>2p+q$ for root gradation reasons.
By \eqref{eq:alphaDalpha} and \eqref{eq:-alphaD-alpha}, $\frac{2}{H(H-1)}\hat h x_{2\al} = x_\al^2$ and $\frac{-2}{(H-2)(H-3)}\hat h x_{-2\al} = x_{-\al}^2$, and therefore,
there is a rational function $g_N(H)$ such that
\begin{align}
\hat h^{2p} z.v_N &= C_N(H,\hat h)g_N(H) x_{-2\al}^px_{-\al}^qx_\al^Nx_{-\al}^N \otimes 1\nonumber \\
&\hspace{-1em}\overset{Lem.\ref{lem:tech}}{=} g_N(H)x_{-2\al}^px_{-\al}^qC_N(H+N,\hat h)F_N(H+n-1,\hat h)F_{N-1}(H+n-2,\hat h)\cdots F_1(H,\hat h)\otimes 1\nonumber \\
&=g_N(H)x_{-2\al}^px_{-\al}^q\otimes 1. C_N(H+N,\hat\la)F_N(H-n-1,\hat \la)\cdots F_1(H,\hat\la),\label{eq:122}
\end{align}
where $\hat\la = (H-1)\la$.
By Lemma \ref{lem:tech}, the product
\[F_N(H_n-1,\hat h)F_{N-1}(H+n-2,\hat h)\cdots F_1(H,\hat h)\]
is an element of $R[\hat h]$ of $\hat h$-degree $2N$. Thus the right hand side of \eqref{eq:122} is nonzero for at least one $\la$. On the other hand, by \eqref{eq:z-vs-HC} 
\[
\hat h^{2p} z.v_N = 0
\]
because $z_0=\varphi(z)=0$. This is a contradiction. Thus $\varphi(z)=0$ implies $z=0$.
\end{proof}

\subsection{Generators of the ghost center}

The following important lemma shows that elements of the image of the Harish-Chandra map satisfy a functional equation. This will reformulated in the subsequent lemma as (relative) invariance under an action of the group $(\Z_{2})^2$.

\begin{Lemma}
Let $z$ be a homogeneous element of the ghost center of ghost parity $\{z\}$, and let $z_0(H,\hat h)=\varphi(z)$ be the image of $z$ in $R[\hat h]$ under the Harish-Chandra homomorphism \eqref{eq:HC}.
Then for every odd integer $n\ge 1$ and every $\epsilon\in\{\pm\}$ the following functional equation holds:
\begin{equation}\label{eq:Zsymmetry}
z_0\big(H+n, \epsilon(H+n-1)(H-1)\big)=(-1)^{\{z\}}z_0\big(H,\epsilon(H+n-1)(H-1)\big).
\end{equation}
\end{Lemma}

\begin{proof}
Let $n$ be an odd positive integer and consider the Verma module $M(\la)$ where $\la=\epsilon(H+n-1)$. Put $\hat\la=\la(H-1)= \epsilon(H+n-1)(H-1)$.
Then, by the explicit formula for $F_n(H,\hat h)$ given in \eqref{eq:Fn-formula}, we have
\[
F_n(H+n-1,\hat\lambda)=0.
\]
With the previous equation and setting $v_\la^{(n)} = x_{-\al}^n.v_\la$, 
we will show $x_{\alpha}.v_{\lambda}^{(n)} = 0 = x_{2\alpha}.v_{\lambda}^{(n)}$. That is to say, $v_\la^{(n)}$ is a singular or primitive vector (a vector annihilated by $Ax_\al+Ax_{2\al}$). 

In $M(\la)$ we have
\begin{align*}
x_\al.v_\la^{(n)} &=
x_\al x_{-\al}^n.v_\la \\
&=F_n(H,\hat h)x_{-\al}^{n-1}.v_\la\\
&=x_{-\al}^{n-1}F_n(H+n-1,\hat h).v_\la\\
&=x_{-\al}^{n-1}F_n(H+n-1,\hat \la).v_\la\\
&=0.
\end{align*}
Since $\hat\la\neq 0$ and \eqref{eq:alphaDalpha} states $x_\al x_\al= \frac{2}{H}hx_{2\alpha} = \frac{2}{H(H-1)}\hat hx_{2\al}$, the above also implies that $x_{2\al}.v_\la^{(n)}=0$. Consequently, we may apply \eqref{eq:z-on-singular} of the Harish-Chandra homomorphism:). 

Let $z$ be a (homogeneous) element of the ghost center and put $z_0(H,\hat h)=\varphi(z)$. Then
\begin{align}
z.v_\la^{(n)} &= z_0(H,\hat h).v_\la^{(n)}\nonumber\\
&= x_{-\al}^n z_0(H+n,\hat h).v_\la\nonumber\\
&= v_\la^{(n)} . z_0\big(H+n,\epsilon (H+n-1)(H-1)\big).\label{eq:z-on-sing1}
\end{align}
On the other hand, since $z$ is ghost central, by \eqref{eq:z-vs-HC} and that $n$ is odd,
\begin{equation}\label{eq:z-on-sing2}
z.v_\la^{(n)} = (-1)^{\{z\}}v_\la^{(n)}.z_0\big(H,\epsilon(H+n-1)(H-1)\big).
\end{equation}
Since $v_\la^{(n)}$ is a nonzero element of the Verma module, \eqref{eq:z-on-sing1} and \eqref{eq:z-on-sing2} imply \eqref{eq:Zsymmetry}.
\end{proof}

\begin{Lemma} \label{lem:ghost-center}
Put $H'=H-1$.
Let $G$ be the order $4$ subgroup of the $\C$-automorphism group of the field $\C(H,h)$ generated by two commuting order two elements $\sigma_+, \sigma_-$ given by
\begin{align*}
\sigma_\epsilon(h) &=\epsilon H',\\
\sigma_\epsilon(H')&=\epsilon h.
\end{align*}
for $\epsilon\in\{\pm\}$.
Let $\chi:G\to \C^\ast$ be the linear group character given by $\chi(\si_+)=\chi(\si_-)=-1$.
Let $\C(H,h)^G$ be the space of $G$-invariants and $\C(H,h)^G_\chi=\{f\in\C(H,h)\mid g(f)=\chi(g)f\;\forall g\in G\}$ be the space of relative invariants with respect to $\chi$.
Let $Z$ (respectively $\antiZ$) be the center (respectively anti-center) of $A$.
Then 
\begin{enumerate}[{\rm (i)}]
\item $\varphi(Z)\subseteq \C(H,h)^G \cap R[\hat h]=\C[(H')^2+h^{2}, H'h]$,
\item $\varphi(\antiZ) \subseteq \C(H,h)^G_\chi\cap R[\hat h]= \C[(H')^2+h^{2}, H'h] \big((H')^2-h^2\big)$.
\end{enumerate}
\end{Lemma}

\begin{proof}
Let $z$ be in the ghost center $\ghostZ(A)$ and $z_0(H,\hat h)=\varphi(z)$.
Fix $\epsilon\in\{\pm\}$ and note that \[\si_\epsilon(H)=1+\si_\epsilon(H')=1+\epsilon h=1+\frac{\epsilon\hat h}{H-1}\]
and
\[\si_\epsilon(\hat h)=\si_\epsilon(H'h)=(\epsilon h)(\epsilon H')=\hat h.\]
Thus we have
\begin{equation}\label{eq:HC-image-lem}
z_0(H,\hat h)-(-1)^{\{z\}}\si_\epsilon\big(z_0(H,\hat h)\big)=
z_0(H,\hat h)-(-1)^{\{z\}}z_0\big(1+
\frac{\epsilon\hat h}{H-1},\hat h\big).
\end{equation}
Substituting $\hat h=\epsilon(H+n-1)(H-1)$, where $n$ is an arbitrary positive odd integer, into \eqref{eq:HC-image-lem} we get
\begin{align*}
&z_0\big(H,\,\epsilon(H+n-1)(H-1)\big)-(-1)^{\{z\}}z_0\big(1+
\epsilon\frac{\epsilon(H+n-1)(H-1)}{H-1},\,\epsilon(H+n-1)(H-1)\big)\\
&=z_0\big(H,\,\epsilon(H+n-1)(H-1)\big)-(-1)^{\{z\}}z_0\big(H+n,\,\epsilon(H+n-1)(H-1)\big)=0
\end{align*}
by \eqref{eq:Zsymmetry}.
Thus, viewed as a polynomial in $\hat h$ with coefficients in $R$, the element $z_0(H,\hat h)-(-1)^{\{z\}}\si_\epsilon\big(z_0(H,\hat h)\big)$ has infinitely many zeros and is thus identically zero (since the coefficient ring $R$ has characteristic zero). This proves the inclusions in (i) and (ii). To prove the equalities, note that $\si_+\si_-(H')=-H'$, $\si_+\si_-(h)=-h$, and thus 
$\C(H',h)^G\subseteq \C\big((H')^{2},H'h,h^{2}\big)$; moreover, since $\chi(\si_+\si_-)=1$, we have the containment $\C(H',h)^G_\chi\subset \C((H')^{2},H'h,h^{2})$, too. Since $G$ is generated by $\si_+\si_-$ and $\si_+$ we now have
\[
\C(H,h)^G =\C\big((H')^2,H'h,h^{2}\big)^{\si_+}=\C\big((H')^{2} + h^{2},H'h\big).
\]
and
\[
\C(H,h)^G_\chi =\C\big((H')^{2},H'h,h^{2}\big)^{\si_+}_\chi\subset\C((H')^{2} + h^{2},H'h,(H')^{2}-h^{2}).
\]
With the reminder that $R[\hat h]=R[h]$, we take intersections with $R[\hat h]$ to yield 
$\C(H,h)^G \cap R[\hat h]=\C[(H')^{2} + h^{2}, H'h]$ and $\C(H,h)^G_\chi\cap R[\hat h]= \C[(H')^{2} + h^{2}, H'h] \big((H')^{2}-h^{2}\big)$. 
\end{proof}

\begin{Lemma}\label{lem:phi-of-C1C2Q2}
We have
\begin{align*}
\varphi\big(C^{(1)}\big) &= 2(H-1)h=2H'h, \\
\varphi\big(C^{(2)}\big) &= (H-1)^2+h^2=(H')^2+h^2,\\
\varphi\big(Q^{(2)}\big) &= (H-1)^2-h^2=(H')^2-h^2.
\end{align*}
\end{Lemma}
\begin{proof}
By the definition of $\varphi$, the values we seek are given by the last two terms (or the only one in the case of $C^{(1)}$) in each of the explicit formulas \eqref{eq:CCnegdiam},\eqref{eq:CCposdiam},\eqref{eq:Qdiam}.
\end{proof}

We are now ready to prove the first main theorem of the paper (Theorem \ref{thm:MainTheorem1}), which describes the ghost center of $A$.

\begin{proof}[Proof of Theorem \ref{thm:MainTheorem1}]
By Lemma \ref{lem:phi-of-C1C2Q2} and Lemma \ref{lem:ghost-center} the Harish-Chandra homomorphism restricts to an isomorphism of $\C$-algebras
\begin{equation}
\varphi\big|_{\ghostZ(A)}:\ghostZ(A)\to \C[x^2+y^2,\, 2xy,\, x^2-y^2] = \C[x^2+y^2,\, 2xy] \oplus
\C[x^2+y^2,\, 2xy](x^2-y^2)
\end{equation}
where $x=H'=H-1$ and $y=h$.
\end{proof}

\begin{Remark}
One can check that the subalgebra of $\C[x,y]$ generated by $x^2+y^2$, $2xy$, and $x^2-y^2$ is isomorphic to $\C[c_1,c_2,q]/\big(q^2-((c_2)^2-(c_1)^2)\big)$ via $2xy\mapsto c_1$, $x^2+y^2\mapsto c_2$, $x^2-y^2\mapsto q$.
\end{Remark}

\section{Finite-dimensional irreducible representations of the diagonal reduction algebra of \texorpdfstring{$\Fosp(1|2)$}{osp(1|2)}}\label{sec:FDreps}

\subsection{Shapovalov form}
In this section, let $\Theta$ denote the (non-super) algebra anti-automorphism of $A$ given by $\Theta( x_{\pm\alpha})= x_{\mp\alpha}$, $\Theta( x_{\pm 2\alpha})=-x_{\mp 2\alpha}$, $\Theta(h)=h$.

For $\la=\la(H)\in R$, let $M(\la)$ be the Verma module from Section \ref{ssec:Verma} and $\Theta\big(M(\la)\big)=({}_\la R)\otimes_{B_-}A$ be the ``opposite'' Verma module, where $B_-$ is the $R$-subring of $A$ generated by $\big\{h, x_{-\al}, x_{-2\al}\big\}$, and ${}_\la R$ is the $(R,B_-)$-bimodule $R$ with right action $1.x_{-\al}=1.x_{-2\al}=0$, $1.h = \la(H)1$. 
By the PBW Theorem for $A$ (Theorem \ref{thm:PBW}), 
\begin{equation}
\Theta\big(M(\la)\big)\otimes_A M(\la)= \left(({}_\la R)\otimes_{B_-} A\right) \otimes_{A} \left(A \otimes_{B} R_\la\right) = ({}_\la R)\otimes_{B_-} A \otimes_{B} R_\la \cong R.
\end{equation}

The \emph{Shapovalov form} on $M(\la)$ is the $\C$-bilinear and $R$-valued form given by the composition of maps
\begin{equation}\label{ShapovalovFormComposition}
\langle\cdot,\cdot\rangle_\la:M(\la)\otimes_\C M(\la) \overset{\Theta\otimes\Id}{\longrightarrow} \Theta\big(M(\la)\big)\otimes_\C M(\la)\longrightarrow 
\Theta\big(M(\la)\big)\otimes_A M(\la)\cong R.
\end{equation}

See \cite[Sections 3.1.4 and 3.1.5]{humphreysRepresentationsSemisimpleLie2008} for a universal construction of the Shapovalov form on the universal enveloping algebra of a semisimple Lie algebra and for a general discussion related to the the validity of the description given below:

\begin{Lemma} \label{lem:shapovalov}
The Shapovalov form satisfies the following properties, for all $u_1,u_2\in M(\la)$ and $f(H)\in R$:
\begin{enumerate}[{\rm (i)}]
\item \label{lem:shapovalovi} $\langle v_\la,v_\la\rangle = 1$;
\item \label{lem:shapovalovii} $\langle f(H).u_1,u_2\rangle= \langle  u_1,f(H).u_2\rangle$;
\item \label{lem:shapovaloviii} $\langle u_1.f(H),u_2\rangle = \langle u_1,u_2\rangle f(H) = \langle u_1,u_2.f(H)\rangle$;
\item \label{lem:shapovaloviv} $\langle x.u_1,u_2\rangle = 
\langle u_1,\Theta(x).u_2\rangle$ for all $x\in A$;
\item \label{lem:shapovalovv} $\langle u_1,u_2\rangle =\langle u_2,u_1\rangle$;
\item \label{lem:shapovalovvi} 
$\langle Rv_{\lambda,p,q},\, Rv_{\lambda,p',q'}\rangle=0$ when $p,p'\in \Z_{\geq 0}, q,q' \in \{0,1\}$, $(p,q)\neq (p',q')$.

\end{enumerate}
\end{Lemma}
Lemma \ref{lem:shapovalov} means $\langle \cdot, \cdot \rangle_{\la}$ is contravariant \cite{khoroshkinContravariantFormReduction2018} and symmetric, in addition, the form yields orthogonal weight spaces of $M(\la)$---as much is suggested by the Shapovalov nomenclature. 
The definition of the radical of the Shapovalov form \eqref{ShapovalovFormComposition} as
\begin{equation}
\op{rad}\langle\cdot,\cdot\rangle_\la = \{v\in M(\la)\mid \langle v,u\rangle=0=\langle u,v\rangle, \forall u\in M(\la)\}
\end{equation}
only differs from the usual case \cite{humphreysRepresentationsSemisimpleLie2008} in that $\langle v, u \rangle = 0 = \langle u, v \rangle$ is viewed an element of $R$. Because \eqref{ShapovalovFormComposition} takes values in the PID $R$, which is not a field, some care is needed in proving the (non-)degeneracy of $\langle \cdot, \cdot \rangle_{\la}$ and in determining the maximal submodules of $M(\la)$. 
\begin{Lemma} \label{lem:shapovalov-more}
Recall that $\hat\la=(H-1)\la$ for $\la\in R$. The Shapovalov form satisfies:
\begin{enumerate}[{\rm (i)}]
\item $\langle x_{-\al}^m.v_\la, x_{-\al}^n.v_\la\rangle = \delta_{m,n}F_1(H,\hat\la)F_2(H+1,\hat \la)\cdots F_n(H+n-1,\hat\la)$.
\item When $\la=0$, $\langle x_{-2\al}^m.v_0, x_{-2\al}^n.v_0\rangle=\delta_{m,n}G_1(H)G_2(H+2)\cdots G_n(H+2n-2)$ where $G_i(H)\in R^\times$ (invertible elements) for all $i>0$.
\item $\langle\cdot,\cdot\rangle_\la$ is nondegenerate if and only if $\lambda=0$ or $F_n(H+n-1,\hat\la)\neq 0$ for all positive integers $n$.
\item If $\langle\cdot,\cdot\rangle_\la$ is degenerate, then
\begin{equation}
\op{rad}\langle\cdot,\cdot\rangle_\la = \sum_{m\ge n} Rx_{-\al}^mv_\la,
\end{equation}
where $n$ is the smallest positive integer such that $F_n(H+n-1,\hat\lambda)=0$. This $n$ is necessarily odd.
\item For every $\mu\in\C\setminus\Z$, define the space $N(\la,\mu) \coloneqq M(\la).(H-1-\mu)+\op{rad}\langle\cdot,\cdot\rangle_\la$, the enclosed $(H-1-\mu)$ being the principal ideal of $R$ generated by $H-1-\mu$. Then we have
\begin{equation}\label{eq:Nlambdamuset}
N(\la,\mu) =\{v\in M(\la)\mid \langle v,u\rangle\in (H-1-\mu)\,\forall u\in M(\la)\}.
\end{equation}
Furthermore, $N(\la,\mu)$ is a maximal $(A,R)$-subbimodule of $M(\la)$, and every maximal $(A,R)$-subbimodule of $M(\la)$ is equal to $N(\la,\mu)$ for some $\mu\in\C\setminus\Z$.
\end{enumerate}
\end{Lemma}

\begin{proof}
(i) By Lemma \ref{lem:shapovalov}(v), without loss of generality we may assume $n\le m$. We have

\begin{align*}
\langle x_{-\al}^m.v_\la, x_{-\al}^n.v_\la\rangle &= \langle x_{-\al}^{m-n}.v_\la,  x_\al^n x_{-\al}^n.v_\la\rangle\\
&\hspace{-0.35em}\overset{\eqref{eq:Fn-congruencebar}}{=}\langle x_{-\al}^{m-n}.v_\la, x_\al^{n-1} F_n(H,\hat h) x_{-\al}^{n-1}.v_\la\rangle\\
&=\langle x_{-\al}^{m-n}.v_\la, x_\al^{n-1}  x_{-\al}^{n-1}F_n(H+n-1,\hat h).v_\la\rangle =\cdots\\
&=\langle x_{-\al}^{m-n}.v_\la, F_1(H,\hat h)F_2(H+1,\hat h)\cdots F_n(H+n-1,\hat h).v_\la\rangle\\
&=\langle x_{-\al}^{m-n}.v_\la, v_\la .F_1(H,\hat \la)F_2(H+1,\hat \la)\cdots F_n(H+n-1,\hat \la)\rangle\\
&=\langle x_{-\al}^{m-n}.v_\la,v_\la\rangle F_1(H,\hat \la)F_2(H+1,\hat \la)\cdots F_n(H+n-1,\hat \la)\\
&=\delta_{m,n} F_1(H,\hat \la)F_2(H+1,\hat \la)\cdots F_n(H+n-1,\hat \la).
\end{align*}

(ii) Proved as in part (i) using Lemma \ref{lem:tech}(i).

(iii) When $\lambda = 0$, the Shapovalov form is nondegenerate by part (ii). 
For $\lambda \neq 0$, \eqref{eq:Mlambdanot0-basis} says $M(\lambda) = \underset{k \in \Z_{\geq 0}}{\bigoplus} R
x_{-\alpha}^{k}v_{\lambda}$, so (i) implies (iii).

(iv) If $\langle\cdot,\cdot\rangle_\la$ is degenerate, by part (iii) we have $\la\neq 0$ and $F_n(H+n-1,\hat\la)=0$ for some positive integer $n$, chosen minimal. 
By Lemma \ref{lem:tech}(ii), $F_{2k}(H+2k-1,\hat\la)$ is a product of nonzero rational functions hence is always nonzero. Therefore, $n$ has to be odd, and the claim follows from part (i).

(v) 
The equality \eqref{eq:Nlambdamuset} between these two sets is shown as follows. 
($\subset$): By Lemma \ref{lem:shapovalov}(iii),  
$\langle v.(H-1-\mu),\,u\rangle_\la = \langle v,\,u\rangle_\la (H-1-\mu)\in (H-1-\mu)$ for any $v,u\in M(\la)$. If $v\in\op{rad}\langle \cdot, \cdot \rangle_{\lambda}$ then $\langle v,u\rangle_\la = 0 \in (H-1-\mu)$ for all $u\in M(\la)$.

($\supset$): The right hand side does not contain $v_\la$, 
by Lemma \ref{lem:shapovalov}(i), since $1$ does not belong to the principal ideal $(H-1-\mu)$ of $R$. Thus the right hand side is a proper $(A,R)$-subbimodule of $M(\la)$.
Hence, equality holds once we show that $N(\la,\mu)$ is a maximal $(A,R)$-subbimodule of $M(\la)$.

It is true that $N(\la,\mu)+N(\la,\mu')=M(\la)$ for any distinct $\mu,\mu'\in\C\setminus\Z$ because the maximal ideals $(H-1-\mu)$ and $(H-1-\mu')$ of $R$ are coprime. Therefore it suffices to show that any maximal $(A,R)$-subbimodule is contained in (hence equal to) $N(\la,\mu)$ for some $\mu\in\C\setminus\Z$, as this automatically shows that each $N(\la,\mu)$ is maximal. (If some $N(\la,\mu)$ were not maximal, then it is contained in some maximal $(A,R)$-subbimodule hence contained in some $N(\la,\mu')$. This forces $\mu'=\mu$ since otherwise $M(\la)=N(\la,\mu)+N(\la,\mu')\subset N(\la,\mu')$ contradicting $N(\la,\mu')$ is proper.)

By Lemma \ref{lem:shapovalov}, the radical is an $(A,R)$-subbimodule of $M(\la)$. Let $N$ be any $(A,R)$-subbimodule of $M(\la)$. Since $M(\la)$ is a weight module with respect to $\op{ad}H$, so is $N$. Since $M(\la)_0=R v_\la=v_\la.R$, the $\la$-weight space of $N$ equals $N_\la=v_\la.J$ for some (possibly zero) proper ideal $J$ of $R$. 

Now let $N$ be a maximal $(A,R)$-subbimodule of $M(\la)$.
Let $S$ be the sum of all $(A,R)$-subbimodules of $N$ such that $S_\la=0$. Then $\langle v_\la, S\rangle_\la =0$ hence $S$ is contained in the radical of $\langle\cdot,\cdot\rangle_\la$ therefore $S$ is contained in $N(\la,\mu)$ for any $\mu$.

Let $T$ be the sum of all $(A,R)$-subbimodules of $N$ whose intersection with $v_\la.R$ is nonzero. 
As discussed above, $T_\la = v_\la.J$ for some nonzero proper ideal $J$ of $R$. We have $J\subset (H-1-\mu)$ for some $\mu\in\C\setminus\Z$, by the weak Nullstellensatz. Now $T\subset M(\la).(H-1-\mu) \subset N(\lambda, \mu)$. So $N=S+T\subset N(\lambda,\mu)$. 
\end{proof}

\subsection{Irreducible highest weight representations}

We put
\[L(\la,\mu) = M(\la)/N(\la,\mu),\qquad \la\in\C, \mu\in\C\setminus\Z.\]

Notice that we now only consider the case when $\la\in\C\subset R$ is a constant function of $H$. Since we specialize $H$ anyway to $\mu+1$, we do not need greater generality to describe finite-dimensional irreducible representations of $A$.

The following two results prove Theorem \ref{thm:MainTheorem2} with added content on infinite-dimensional irreducible $A$-representations.  
\begin{Lemma}\label{lem:action-of-ghost-center}
The action of the ghost center on $L(\la,\mu)$ is given by
\begin{equation}
C^{(1)} \mapsto 2\la\mu,\quad C^{(2)}\mapsto \mu^2+\lambda^2,\quad Q^{(2)}\mapsto (\mu^2-\lambda^2)(-1)^{|\cdot|},
\end{equation}
where $(-1)^{|\cdot|}\in\End_\C\big(L(\la,\mu)\big)$ denotes the parity sign function defined on homogeneous vectors by $v\mapsto (-1)^{|v|}$. 
\end{Lemma}

\begin{proof}
Since $L(\lambda, \mu)$ is a quotient of $M(\lambda, \mu)$, we may extend part (ii) of Lemma \ref{lem:AmoduleAction} and use Lemma \ref{lem:phi-of-C1C2Q2}.

\end{proof}

\begin{Theorem} \label{thm:L-lambda-mu}
The following statements hold:
\begin{enumerate}[{\rm (i)}]
\item $L(\la,\mu)$ is an irreducible representation of $A$, for any $(\la,\mu)\in\C\times(\C\setminus\Z)$.
\item $L(\la,\mu)$ is finite-dimensional if and only if
\begin{equation}\label{eq:la-mu-dim}
\la^2=(\mu+n)^2
\end{equation}
for some odd positive integer $n$. In this case, $n=\dim L(\la,\mu)$ and $\la \neq 0$.
\item Every finite-dimensional irreducible representation $V$ of $A$ is odd-dimensional and isomorphic to $L(\la,\mu)$ for a unique pair $(\la,\mu)\in\C\times(\C\setminus\Z)$ satisfying \eqref{eq:la-mu-dim}, where $n=\dim V$.
\end{enumerate}
\end{Theorem}

\begin{proof}
(i) We need to show that $N(\la,\mu)$ is a maximal $A$-submodule of $M(\la)$. Let $W$ be an $A$-submodule of $M(\la)$ properly containing $N(\la,\mu)$. It suffices to show that $W$ is a right $R$-submodule of $M(\la)$, because then $W=M(\la)$ by Lemma \ref{lem:shapovalov-more}(v). Let $f(H)\in R$ and $w\in W$. Then $w.f(H) - w.f(\mu+1) \in N(\la,\mu)$ because $f(H)-f(\mu+1)$ is divisible by $H-1-\mu$. But $w.f(\mu+1)=f(\mu+1)w\in W$ since $f(\mu+1)\in\C$. So $w.f(H)\in N(\la,\mu)+W=W$.

(ii) Suppose $L(\la,\mu)$ is finite-dimensional. Denote the image of $v_\la\in M(\la)$ in $L(\la,\mu)$ by $\bar v_\la$. Suppose $\la=0$. Then $\hat \la=(H-1)\la=0$. Thus $\hat h.\bar v_\la = 0$. Therefore, by Lemma \ref{lem:tech}(i), $\hat x_{-2\al}^n. \bar v_\la$ must be nonzero for all $n>0$ (since $\mu\notin\Z$). This contradicts that $L(\la,\mu)$ is finite-dimensional.

Thus we may assume $\la\neq 0$.
Leveraging the assumption that $L(\la,\mu)$ is finite-dimensional, it follows from $\{x_{-\alpha}^{k}.\bar{v}_{\lambda}\}_{k \in \Z_{\geq 0}} $ being a set of $H$-weight vectors of distinct weights that $x_{-\al}^n \bar v_\la =0$ for some, minimally chosen, non-negative integer $n$. Since $\la\neq 0$, the relations in $A$ imply that $\{x_{-\al}^i\bar v_\la\}_{i=0}^{n-1}$ is a basis for $L(\la,\mu)$ hence $n=\dim L(\la,\mu)$. 
By Lemma \ref{lem:tech}(ii) we have
$F_n(H,\hat h) x_{-\al}^{n-1}. \bar v_\la=0$.
Thus $x_{-\al}^{n-1}\bar v_\la.F_n(H+n-1,\hat \la) =0$.
As in the proof of part (i), we can reduce $H$ on the right modulo the ideal $(H-1-\mu)$ of $R$. That is, remembering that $\hat\la=(H-1)\la$, we get $x_{-\al}^{n-1}\bar v_\la.F_n(\mu+n,\mu \la)=0$, or, rewritten, $F_n(\mu+n,\mu\la) x_{-\al}^{n-1}. \bar v_\la=0$. Since $n$ was minimal, this forces $F_n(\mu+n,\mu\la)=0$.
By Lemma \ref{lem:tech}(ii), $n$ must be odd and $(\mu+n)^2\mu^2-(\mu\la)^2=0$. Since $\mu\notin\Z$, this implies $(\mu+n)^2=\la^2$. 

Conversely, if $(\mu+n)^2=\la^2$ for some odd positive integer $n$, then $\la\neq 0$ and we may retrace the above steps to conclude that $x_{-\al}^n \bar v_\la =0$ and $n$ is minimal with this property, which implies that $L(\la,\mu)$ is $n$-dimensional.

(iii) Since $V$ is finite-dimensional, the relations
$Hx_{k\al}=x_{k\al}(H-k)$, for $k=1,2$, and $x_{2\al}x_\al \in R^\times x_\al x_{2\al}$ imply that there exists a nonzero vector $v\in V$ such that
\begin{equation}\label{eq:v-condition}
(H-1).v=\mu v,\qquad x_\al. v=0,\qquad x_{2\al}.v=0
\end{equation}
for some $\mu\in\C$. Since $H-n$ is invertible in $R$ for all $n\in\Z$, it must act invertibly on $V$, and therefore $\mu\notin\Z$.
Since $h x_{k\al} \in R^\times x_{k\al} h$ and $h$ commutes with $H$, it follows that $h$ preserves the subspace of all $v\in V$ satisfying \eqref{eq:v-condition}. Thus there exists $v_0\in V$ such that
\[
h.v_0=\lambda v_0,\qquad (H-1).v_0=\mu v_0,\qquad x_\al. v_0=0,\qquad x_{2\al}.v_0=0,
\]
for some $(\la,\mu)\in\C\times(\C\setminus\Z)$.

By the PBW Theorem (Theorem \ref{thm:PBW}), the set $\big\{x_{-2\al}^k x_{-\al}^j. v_0\mid k\ge 0, j\in\{0,1\}\big\}$ spans $Av_0$ over $\C$ which equals to $V$ (by irreducibility of $V$ and that $Av_0$ is a nonzero subrepresentation).
We can turn $V$ into a right $R$-module by defining $v.(H-1)=\mu v$ for all $v\in V$.
By the universal property of $M(\la)$ coming from its definition, there exists a unique surjective map of $(A,R)$-bimodules $M(\la)\to V$ sending $v_\la$ to $v_0$. The kernel is a maximal $(A,R)$-subbimodule of $M(\la)$. By Lemma \ref{lem:shapovalov-more}(iv), $V\cong M(\la)/N(\la,\mu)=L(\la,\mu)$ for some $\mu\in\C\setminus\Z$.

It remains to prove uniqueness of $(\la,\mu)$. Suppose $L(\la,\mu)\cong L(\la',\mu')$ where $\la,\la'\in\C$ and $\mu,\mu'\in\C\setminus\Z$ satisfy
\begin{equation}\label{eq:pf-irreps-1}
\la^2 = (\mu+n)^2,\qquad (\la')^2=(\mu'+n)^2
\end{equation}
where $n=\dim L(\la,\mu)=\dim L(\la',\mu')$. Since the representations are isomorphic, the actions of $C^{(1)}, C^{(2)}, Q^{(2)}$ must yield the same scalars. Thus, by Lemma \ref{lem:action-of-ghost-center},
\begin{align}
2\la\mu &=2\la'\mu', \label{eq:pf-irreps-2}\\
\la^2+\mu^2 &= (\la')^2+(\mu')^2, \\
\la^2-\mu^2 &= (\la')^2-(\mu')^2.\label{eq:pf-irreps-4}
\end{align}
Expanding the squares and subtracting the two equations in \eqref{eq:pf-irreps-1} from each other using \eqref{eq:pf-irreps-4} we obtain $2(\mu-\mu')n=0$. Since $n$ is an odd integer, $\mu=\mu'$. Since $\mu$ is not an integer hence nonzero, \eqref{eq:pf-irreps-2} gives $\la=\la'$.
\end{proof}

\section{Tensor products}\label{sec:TensorProducts}

Let $V(\xi_i)$ be two irreducible highest weight representations of $\Fosp(1|2)$ of highest weights $\xi_i\in\C$, $i\in\{1,2\}$, such that $\xi_1+\xi_2\notin\Z$. Let $v_i$ denote the highest weight vector of $V(\xi_i)$. To establish the following examples in their clarity we return to using bars. Recall from Section \ref{sec:DRAosp} that $\tilde h = h\otimes 1 - 1\otimes h$ and $H = h \otimes 1 + 1 \otimes h$. Now $(H-1) v_1\otimes v_2 = (\xi_1+\xi_2-1) v_1\otimes v_2$. So \[\mu=\xi_1+\xi_2-1.\]
Set $V = V(\xi_{1}) \otimes V(\xi_{2})$. The action of $\bar h$ on the space $V^+$ of primitive vectors in $V$ is given by $\bar h.w = (P\tilde h+I).w$, for $w \in V^{+}$; here, $P=\sum_{n=0}^\infty \varphi_n(H) X_{-\al}^n X_\al^n$ is the extremal projector for the image of $\Fosp(1|2)$ in $\Fosp(1|2)\times\Fosp(1|2)$ under the diagonal embedding $\delta$.
\begin{Proposition}[{\cite[Proposition 3.1]{hartwigDiagonalReductionAlgebra2022a}}]
\label{prp:tildestuff}
	The algebra $Z\left(\Fosp(1|2) \times \Fosp(1|2) , \Fosp(1|2)\right)$ is generated as a $D^{-1}U(\Fh)$-ring by the following elements:
	\begin{align*}
		P\tilde x_{2\al} +I &= \tilde x_{2\al} + I \\
		P\tilde x_\al +I&= \tilde x_\al -2 \varphi_1(H) X_{-\al} \tilde x_{2\al}+ I\\
		P\tilde{h}+I &= \tilde h + \varphi_1(H) X_{-\al} \tilde x_\al-2 \varphi_2(H) X_{-\al}^2 \tilde x_{2\al}+I\\
		P\tilde x_{-\al} +I&=\tilde x_{-\al} + \varphi_1(H)X_{-\al}\tilde h+\varphi_2(H)X_{-\al}^2\tilde x_\al-2\varphi_3(H)X_{-\al}^3\tilde x_{2\al}+I\\
		P\tilde x_{-2\al} +I&= \tilde x_{-2\al}+\varphi_1(H)X_{-\al}\tilde x_{-\al}+\varphi_2(H)X_{-\al}^2\tilde h+\varphi_3(H)X_{-\al}^3\tilde x_\al -2\varphi_4(H)X_{-\al}^4\tilde x_{2\al}+I
	\end{align*}
\end{Proposition}
By this proposition, and that $\tilde x_{\al}.v_1\otimes v_2=0$ in the special situation where $v_i$ are the highest weight vectors of the respective factors, we have
\begin{equation}\label{eq:barh-onv1v2}
\bar h.v_1\otimes v_2 = (P\tilde h+I).v_{1}\otimes v_{2} = \tilde h.v_1\otimes v_2= (\xi_1-\xi_2) v_1\otimes v_2.
\end{equation}
So \[\lambda=\xi_1-\xi_2.\]

\begin{Example}[$\xi_1=1/2$, $\xi_2=-1$]
Consider $V(1/2)\cong\C[x]$, $V(-1)\cong \C^{1|2}$. Here $\la=\frac{1}{2}-(-1)=\frac{3}{2}$ and $\mu=\frac{1}{2}+(-1)-1=\frac{-3}{2}$. The action of (S)Casimirs of the diagonal reduction algebra on the space $V^+$ of primitive vectors in $V=V(1/2)\otimes V(-1)$ is given by
\begin{align*}
C^{(1)}&\mapsto 2\lambda\mu=2\left(\frac{3}{2}\right)\left(-\frac{3}{2}\right)=-\frac{9}{2} \tag{see Example 4.2 in \cite{hartwigDiagonalReductionAlgebra2022a}}
\\
C^{(2)}&\mapsto \mu^2+\lambda^2 = \frac{9}{4}+\frac{9}{4}=\frac{9}{2}\\ 
Q^{(2)}&\mapsto (\mu^2-\lambda^2)(-1)^{|\cdot|}  = (\frac{9}{4}-\frac{9}{4})(-1)^{|\cdot|}=0 \tag{see Example 4.3 in \cite{hartwigDiagonalReductionAlgebra2022a}}
\end{align*}
From \eqref{eq:la-mu-dim}, we also obtain that $\dim V^{+} = -\mu\mp\la \in \{0,3\}$; necessarily, we take $+$ for $\mp$ and arrive at $\dim V^{+} = 3$ (see  \cite[Example 3.1]{hartwigDiagonalReductionAlgebra2022a} for explicit expressions of a basis of $V^{+}$).

Now $\bar{\mathcal{C}}_{-}$ acts by $\frac{1}{4} \la\mu = -\frac{9}{16}$ on $V^{+}$, and $C\otimes 1-1\otimes C$ acts by $-\frac{9}{16}$ on $V$;  likewise, $\bar{\mathcal{C}}_{+}$ acts by $\frac{1}{8} (\la^2+\mu^2) = \frac{9}{16}$ on $V^{+}$, and $C\otimes 1 + 1\otimes C$ acts by $\frac{9}{16}$ on $V$.
\end{Example}

\begin{Example}[$\xi_1=1/2$, $\xi_2=-\ell\in \Z_{<0}$]\label{ex:FiveThree}
We now consider $V=V(1/2)\otimes V(-\ell)$,
generalizing the previous example, with $V(\xi_1)\cong\C[x]$ and $V(-\ell)$ as the unique finite-dimensional irreducible representation of $\Fosp(1|2)$ of dimension $2\ell+1$.

Here $\la=\frac{1}{2}+\ell$ and $\mu=\frac{1}{2}-\ell-1.$
Again by \eqref{eq:la-mu-dim}, the space $V^+$ of primitive vectors in $\C[x]\otimes V(-\ell)$ has dimension
\begin{equation}\label{eq:primDimPolyTensor}
\dim V^{+} = -\mu\mp\la = -(\frac{1}{2}-\ell-1) \mp (\frac{1}{2}+\ell). 
\end{equation}
If $\mp=-$, then Equation \eqref{eq:primDimPolyTensor} is $0$, in contradiction with the existence of the primitive vector $v_{1} \otimes v_{2}$ formed from the nonzero $v_{1}, v_{2}$.

So we must take $+$ for $\mp$ in Equation \eqref{eq:primDimPolyTensor}, which gives
\[\dim V^{+} = -(\frac{1}{2}-\ell-1) + (\frac{1}{2}+\ell) =2\ell+1 = \dim V(-\ell).\]
\end{Example}

Recall that $V^{+}$ is an irreducible $A$-representation \cite[Proposition 2.2]{hartwigDiagonalReductionAlgebra2022a} upon which $(H-n)$, $n \in \Z$, acts invertibly. This fact plays a role in proving Theorem \ref{thm:MainTheorem3}. 
\begin{proof}[Proof of Theorem \ref{thm:MainTheorem3}] 
Let $v_\ell$ be a highest weight vector in $V(-\ell)$, characterized by $h.v_{\ell}=-\ell v_\ell$ and $x_\al.v_\ell=0$. Then $w=1\otimes v_\ell$ is a primitive vector in $V = \C[x]\otimes V(-\ell)$. We have from Example \ref{ex:FiveThree} that $(H-1).w = \mu w$ with $\mu=\frac{1}{2}-\ell-1$; and, 
$\bar h.w = \lambda w$, with $\lambda=\frac{1}{2}+\ell$.
Since $\lambda \neq 0$, we use the same reasoning of \eqref{eq:Mlambdanot0-basis} to determine that a basis for the space of primitive vectors $V^+$ is the set
$\{ \bar x_{-\al}^j.w\mid 0\le j\le 2\ell\}$.
As in the discussion in the beginning of this section, we use the characterization of $w \in V^{+}$ to write the action
$\bar x_{-\al}^j.w = (P\tilde x_{-\al}+I)^j.w=(P\tilde x_{-\al})^j.w$. By Proposition \ref{prp:tildestuff} we arrive at
\begin{equation}
\bar x_{-\al}^j.w = 
\big( \tilde x_{-\al} + \varphi_1(H)X_{-\al}\tilde h+\varphi_2(H)X_{-\al}^2\tilde x_\al-2\varphi_3(H)X_{-\al}^3\tilde x_{2\al}\big)^j.w,
\end{equation}
and after substitutions for $\varphi_{1}(H), \varphi_{2}(H), \varphi_{3}(H)$  made via \cite[Section 3.2]{hartwigDiagonalReductionAlgebra2022a}, the result is \eqref{eq:polytensordecomposition}:
\begin{equation*}
 \C[x]\otimes V(-\ell) = \bigoplus_{j=0}^{2\ell} U(\Fn_-).\big( \tilde x_{-\al} - \frac{1}{H-1}X_{-\al}\tilde h-\frac{1}{H-1}X_{-\al}^2\tilde x_\al-\frac{2}{(H-2)(H-1)}X_{-\al}^3\tilde x_{2\al}\big)^j.(1\otimes v_\ell).
\end{equation*}
\end{proof}

\renewcommand\appendix{\par
 \setcounter{section}{0}%
  \renewcommand{\thesection}{\Alph{section}}
}

\renewcommand{\theHsection}{A\arabic{section}}

\appendix

\section{Proof of technical lemma}\label{appendix:A}

In this appendix we prove Lemma \ref{lem:tech}.
To simplify notation, we one again omit bars over the variables. Thus we write $x_\be$ for the element $\bar x_\be=x_\be\otimes 1-1\otimes x_\be+\textup{\Romanbar{II}}$ in the diagonal reduction algebra $U/\textup{\Romanbar{II}}$.

\begin{Lemma}
$F_n(H,\hat h)$ satisfies the difference-recursion relation
\begin{equation}\label{eq:66}
F_{n+1}(H,\hat h) = \frac{-H}{H-1}F_n(H-1,\hat h) + \frac{-H(H-2)(H-3)^2}{H-1}\frac{1}{\hat h^2} F_n(H-1,\hat h)F_{n-1}(H-2,\hat h) - \frac{1}{H(H-1)^2}\hat h^2 + H
\end{equation}
with initial conditions 
\begin{equation}\label{eq:66-init}
F_0(H,\hat h)=0,\qquad F_1(H,\hat h)=H-\frac{1}{H(H-1)^2}\hat h^2.
\end{equation}
\end{Lemma}

\begin{proof}
We prove this by induction on $n$. For $n=0$, the result holds. For $n=1$, we have by relation \eqref{eq:alphaD-alpha}
\begin{equation}
x_\al x_{-\al} = -\frac{H}{H-1} x_{-\al}x_\al + \frac{4H}{(H-1)(H-2)}x_{-2\al}x_{2\al} - \frac{1}{H(H-1)^2}\hat h^2+H,
\end{equation}
which is congruent to $H-\frac{1}{H(H-1)^2}\hat h^2$ modulo $Ax_\al+Ax_{2\al}$.
Suppose that \eqref{eq:66} holds for some arbitrary $n\ge 1$. Then we have:
\begin{align}
x_\al x_{-\al}^{n+1} &\overset{\eqref{eq:alphaD-alpha}}{\equiv}
\Big(-\frac{H}{H-1} x_{-\al}x_\al + \frac{4H}{(H-1)(H-2)}x_{-2\al}x_{2\al} - \frac{1}{H(H-1)^2}\hat h^2+H
\Big) x_{-\al}^n\nonumber \\
&\hspace{-0.60em}\overset{\text{ind. hyp.}}{\equiv}
-\frac{H}{H-1}x_{-\al} F_n(H,\hat h) x_{-\al}^{n-1} +
\frac{4H}{(H-1)(H-2)}x_{-2\al}x_{2\al} x_{-\al}^n + (-\frac{1}{H(H-1)^2}\hat h^2+H)x_{-\al}^n.
\label{eq:59}
\end{align}
By \eqref{eq:alphaDalpha},
\begin{equation}\label{eq:60}
x_\al^2 = \frac{2\hat h}{H(H-1)} x_{2\al}.
\end{equation}
Thus,
\begin{align}
\frac{2\hat h}{H(H-1)} x_{2\al} x_{-\al}^n &\overset{\eqref{eq:60}}{=} x_\al^2x_{-\al}^n \nonumber\\
&\overset{\eqref{eq:59}}{\equiv} x_\al F_n(H,\hat h) x_{-\al}^{n-1} \nonumber\\
&\overset{\eqref{eq:xf}}{=} F_n(H+1,\hat h)x_\al x_{-\al}^{n-1} \nonumber\\
&\hspace{-.60em}\overset{\text{ind. hyp.}}{\equiv} F_n(H+1,\hat h) F_{n-1}(H,\hat h)x_{-\al}^{n-2}.\label{eq:61}
\end{align}
Note that \eqref{eq:61} should be interpreted as zero when $n=1$.
Applying $\Theta$ to \eqref{eq:60} we get
\begin{equation}\label{eq:63}
x_{-\al}x_{-\al} = \frac{-2\hat h}{(H-2)(H-3)}x_{-2\al}.
\end{equation}
Combining these facts we have
\begin{align}
\frac{-4\hat h^2}{(H-2)^2(H-3)^2} x_\al x_{-\al}^{n+1}
&\hspace{-.35em}\overset{\eqref{eq:59},\eqref{eq:61}}{\equiv}
\frac{-4\hat h^2}{(H-2)^2(H-3)^2} \frac{-H}{H-1} F_n(H-1,\hat h) x_{-\al}^n \nonumber\\
&\qquad+
\frac{4H}{(H-1)(H-2)}\frac{-2\hat h}{(H-2)(H-3)}x_{-2\al}F_n(H+1,\hat h)F_{n-1}(H,\hat h) x_{-\al}^{n-2} \nonumber\\&\qquad+ \frac{-4\hat h^2}{(H-2)^2(H-3)^2} (-\frac{1}{H(H-1)^2}\hat h^2+H)x_{-\al}^n
\nonumber\\
&\hspace{.75em}\overset{\eqref{eq:63}}{=}
\Big[
\frac{-4\hat h^2}{(H-2)^2(H-3)^2} \frac{-H}{H-1} F_n(H-1,\hat h) \nonumber\\
&\qquad+
\frac{4H}{(H-1)(H-2)}F_n(H-1,\hat h)F_{n-1}(H-2,\hat h) \nonumber\\&\qquad+ \frac{-4\hat h^2}{(H-2)^2(H-3)^2} (-\frac{1}{H(H-1)^2}\hat h^2+H)\Big]x_{-\al}^n.\label{eq:64}
\end{align}
By the PBW theorem, \eqref{eq:64} implies that $\hat h^2$ divides $F_n(H-1,\hat h)F_{n-1}(H-2,\hat h)$. Using that $x_\al x_{-\al}^{n+1}\equiv F_{n+1}(H,\hat h)x_{-\al}^n$, \eqref{eq:64} implies \eqref{eq:66}.
\end{proof}

\begin{Lemma}
The degree of $F_n(H,\hat h)$ as a polynomial in $\hat h$ with coefficients in $R$, is at most two.
\end{Lemma}

\begin{proof}
For $n=0,1$, the claim is true by \eqref{eq:66-init}. 
For $n>1$, it follows by induction on $n$, using the recursion relation \eqref{eq:66}.
\end{proof}

Write
\begin{equation}\label{eq:F2}
F_n(H,\hat h) = c_n^0(H)+c_n^1(H)\hat h+c_n^2(H)\hat h^2.
\end{equation}

\begin{Lemma}\label{lem:c1}
$c_n^1(H)=0$ for all $n\ge 0$.
\end{Lemma}

\begin{proof}
The statement is true for $n=0,1$ by \eqref{eq:66-init}. Identifying the coefficients of $\hat h$ in both sides of \eqref{eq:66} we obtain
\begin{equation}
c_{n+1}^1(H)=\frac{-H}{H-1}c_n^1(H-1)+\frac{-H(H-2)(H-3)^2}{H-1}\Big(c_n^1(H-1)c_{n-1}^2(H-2)+c_n^2(H-1)c_{n-1}^1(H-2)\Big).
\end{equation}
Thus the claim follows by induction on $n$.
\end{proof}

\begin{Lemma}\label{lem:c2}
For $n\ge 0$ we have
\begin{equation}\label{eq:c2}
c_n^2(H)=\begin{cases}
\displaystyle \frac{1}{H(H-1)^2}\Big(\frac{H^2}{(H-n)^2}-1\Big),&\text{$n$ even,}\\
\displaystyle \frac{-1}{H(H-1)^2},&\text{$n$ odd.}
\end{cases}
\end{equation}
\end{Lemma}

\begin{proof}
Identifying the coefficients of $\hat h^2$ in both sides of \eqref{eq:66} we obtain
\begin{equation}\label{eq:c2pf}
c_{n+1}^2(H)=\frac{-H}{H-1}c_n^2(H-1)+\frac{-H(H-2)(H-3)^2}{H-1}c_n^2(H-1)c_{n-1}^2(H-2)-\frac{1}{H(H-1)^2}.
\end{equation}
Suppose $n$ is odd. By \eqref{eq:66-init} we have $c_1^2(H)=\frac{-H}{H(H-1)^2}$. Assume that \eqref{eq:c2} holds for $n=2k-1$ where $k>0$. Then, taking $n=2k$ in \eqref{eq:c2pf}, we get
\begin{align*}
c_{2k+1}^2(H)&=\frac{-H}{H-1}c_{2k}^2(H-1)+\frac{-H(H-2)(H-3)^2}{H-1}c_{2k}^2(H-1)\frac{-1}{(H-2)(H-3)^2}-\frac{1}{H(H-1)^2}\\
&=-\frac{1}{H(H-1)^2}.
\end{align*}
By induction, \eqref{eq:c2} holds for all odd $n$. Suppose $n$ is even. For $n=0$, we have $F_0(H,\hat 0)=0$ so $c_0^2(H)=0$, in agreement with \eqref{eq:c2}. Taking $n=2k+1$ in \eqref{eq:c2pf} and using that $c_{2k+1}^2(H)=\frac{-1}{H(H-1)^2}$ we get
\[
c_{2k+2}^2(H)=\frac{-H}{H-1}\frac{-1}{(H-1)(H-2)^2}+\frac{-H(H-2)(H-3)^2}{H-1}\frac{-1}{(H-1)(H-2)^2}c_{2k}^2(H-2)-\frac{1}{H(H-1)^2}
\]
Multiplying both sides by $H(H-1)^2$ and putting $\hat c_n^2(H)=H(H-1)^2c_n^2(H)$ for all $n$ we get
\[
\hat c_{2k+2}^2(H) = \frac{H^2}{(H-2)^2}+\frac{H^2}{(H-2)^2} \hat c_{2k}^2(H-2)-1
\]
or, equivalently,
\[
\hat c_{2k+2}^2(H)+1 = \frac{H^2}{(H-2)^2}\big(\hat c_{2k}^2(H-2)+1\big).
\]
Iterating this formula and using that $c_0^2(H)=0$, we obtain that for all $k\ge 0$,
\[
\hat c_{2k}^2(H)+1=\frac{H^2}{(H-2k)^2}.
\]
Thus for all $k\ge 0$,
\[
c_{2k}^2(H) = \frac{1}{H(H-1)^2}\Big(\frac{H^2}{(H-2k)^2}-1\Big),
\]
proving \eqref{eq:c2} for even $n$.
\end{proof}

It remains to determine the constant term $c_n^0(H)$ of $F_n(H,\hat h)$.

\begin{Lemma}\label{lem:c0}
\begin{equation}\label{eq:c0}
c_n^0(H)=
\begin{cases}
\displaystyle 0,& \text{$n$ even,}\\
\displaystyle H\frac{(H-n)^2}{(H-1)^2},&\text{$n$ odd.}
\end{cases}
\end{equation}
\end{Lemma}

\begin{proof}
Taking the constant term in both sides of \eqref{eq:66} we get
\begin{equation}\label{eq:c0-pf}
c_{n+1}^0(H)=\frac{-H}{H-1}c_n^0(H-1)+\frac{-H(H-2)(H-3)^2}{H-1}\big(c_n^0(H-1)c_{n-1}^2(H-2)+c_n^2(H-1)c_{n-1}^0(H-2)\big)+H.
\end{equation}
Suppose first that $n=2k$. Then, using \eqref{eq:c2}, we get
\begin{align*}
c_{2k+1}^0(H)&=
\frac{-H}{H-1}c_{2k}^0(H-1)+\frac{-H(H-2)(H-3)^2}{H-1}\Big[c_{2k}^0(H-1)\frac{-1}{(H-2)(H-3)^2}\\ &\quad+\frac{1}{(H-1)(H-2)^2}\Big(\frac{(H-1)^2}{(H-1-2k)^2}-1\Big)c_{2k-1}^0(H-2)\Big]+H\\
&=\frac{-H(H-3)^2}{(H-1)^2(H-2)}\Big(\frac{(H-1)^2}{(H-1-2k)^2}-1\Big)c_{2k-1}^0(H-2)+H
\end{align*}
which, after multiplying by $(H-1)^2/H$ and putting $\tilde c_n^0(H)=\frac{(H-1)^2}{H}c_n^0(H)$ for all $n$, can be written
\begin{equation}\label{eq:c0-odd}
\tilde c_{2k+1}^0(H) = \Big(1-\frac{(H-1)^2}{\big(H-(2k+1)\big)^2}\Big)\tilde c_{2k-1}^0(H-2)+(H-1)^2.
\end{equation}
We use this to show by induction on $k$ that
\begin{equation}\label{eq:tilde-c0}
\tilde c_{2k+1}^0(H)=\big(H-(2k+1)\big)^2.
\end{equation}
Indeed, by \eqref{eq:66-init}, $c_1^0(H)=H$ so that $\tilde c_1^0(H)=(H-1)^2$, which proves \eqref{eq:tilde-c0} for $k=0$. For $k>0$,
\begin{align*}
\tilde c_{2k+1}^0(H)&\overset{\eqref{eq:c0-odd}}{=} \Big(1-\frac{(H-1)^2}{\big(H-(2k+1)\big)^2}\Big)\tilde c_{2k-1}^0(H-2)+(H-1)^2\\
&\hspace{-.25em}\overset{\text{ind.hyp.}}{=}\Big(1-\frac{(H-1)^2}{\big(H-(2k+1)\big)^2}\Big)\big(H-2-(2k-1)\big)^2 + (H-1)^2 \\
&\hspace{0.74em}= \big(H-(2k+1)\big)^2.
\end{align*}
Consequently, for all odd $n$,
\[
c_n^0(H)=\frac{H}{(H-1)^2}\tilde c_n^0(H)=H\frac{(H-n)^2}{(H-1)^2}
\]
as claimed.

That $c_{2k}^0(H)=0$ for all $k\ge 0$ follows from the fact that $\hat h^2$ divides $F_{2k+1}(H-1,\hat h)F_{2k}(H-2,\hat h)$ for all $k\ge 0$. Alternatively, take $n=2k+1$ in \eqref{eq:c0-pf} and use that $c_{2k+1}(H)=H\frac{(H-(2k+1))^2}{(H-1)^2}$ to get
\begin{align*}
c_{2k+2}^0(H)&=\frac{-H}{H-1}c_{2k+1}^0(H-1)+\frac{-H(H-2)(H-3)^2}{H-1}\Big[c_{2k+1}^0(H-1)c_{2k}^2(H-2)\\ &\quad+c_{2k+1}^2(H-1)c_{2k}^0(H-2)\Big]+H.
\end{align*}
After substituting \eqref{eq:c0} for $n=2k+1$ and \eqref{eq:c2} for $n=2k$ and $n=2k+1$, respectively, the previous equation simplifies to
\begin{equation}
c_{2k+2}^0(H) = \frac{H(H-3)^2}{(H-1)^2(H-2)}c_{2k}^0(H-2),
\end{equation}
which, together with the fact that $c_0^0(H)=0$, inductively shows that $c_{2k}^0(H)=0$ for all $k\ge 0$.
\end{proof}

Substituting \eqref{eq:F2} into \eqref{eq:Fn-congruencebar} using Lemmas \ref{lem:c1},\ref{lem:c2},\ref{lem:c0}, we obtain the proof of Lemma \ref{lem:tech}.

\printbibliography
\end{document}